\documentclass[11pt,reqno]{amsart}
\usepackage[a4paper,top=2.4cm, bottom=2.4cm,left=2.4cm, right=2.4cm,twoside]{geometry}

\usepackage{amsthm}
 \usepackage{amsmath}
 \usepackage{amsfonts}
 \usepackage{amssymb}
 \usepackage{amscd}
 \usepackage{url}
 \usepackage[T1]{fontenc}
 \usepackage[utf8x]{inputenc}
 \usepackage[english]{babel}
 \usepackage[mathscr]{eucal}
 \usepackage{lmodern}
 \usepackage{bera}
 \usepackage{float} 
 \usepackage{bold-extra}
 \usepackage{textcomp}
 \usepackage{graphicx}
 \usepackage{caption}
\usepackage{subcaption} 
\usepackage{multicol}
 \usepackage[dvipsnames,table]{xcolor}
 \usepackage{mathtools} 
 \usepackage{enumitem}
 \usepackage{tikz} 
 \usepackage{tikz-cd}
 \usepackage{wasysym}
 \usepackage{todonotes}
 \usepackage{fancybox}
 \usepackage{fancyhdr}
 \usepackage{breqn}
 \usepackage[nodisplayskipstretch]{setspace}
\usepackage{hyperref}
\hypersetup{
	colorlinks=true,
	linkcolor=blue,
	filecolor=darkmagenta,      
	urlcolor=black,
	citecolor=red
}

 \theoremstyle{definition}  
  \newtheorem{defn}{Definition}[section]
  \newtheorem{eg}[defn]{Example}

   \newtheorem{rmk}[defn]{Remark}

  \theoremstyle{plain}  
  \newtheorem{thm}[defn]{Theorem}
  \newtheorem{lem}[defn]{Lemma}
  \newtheorem{prop}[defn]{Proposition}
  \newtheorem{cor}[defn]{Corollary}
  \newtheorem{conj}[defn]{Conjecture}
  \newtheorem{que}[defn]{Question}

  \theoremstyle{remark}

 \numberwithin{equation}{section}
 \allowdisplaybreaks[4] 
 \setlist[enumerate]{font=\upshape,noitemsep, topsep=0pt} 
 \setlist[itemize]{noitemsep, topsep=0pt}

\makeatletter
\@namedef{subjclassname@2020}{\textup{2020} Mathematics Subject Classification}
\makeatother

\title{Matricial ranges, dilations, and unital contractive maps}

\author{Pankaj Dey}
\address{Pankaj Dey, Department of Mathematics, Indian Institute of Technology Bombay, Powai, Mumbai 400076, India.}
\email{pankajdey2022@gmail.com, pankaj@math.iitb.ac.in}
\address{Current address: Statistics and Mathematics Unit, Indian Statistical Institute, Bangalore Centre, 8th Mile Mysore Road, RVCE Post, Bangalore 560059, India.}

 \author{Atanu Dhang}
\address{Atanu Dhang, School of Mathematical and Computational Sciences, Indian Association for the Cultivation of Science, 2A and 2B Raja S C Mullick Road, Jadavpur, Kolkata-700032, India.}
 \email{atanudhang@gmail.com, smcsad2713@iacs.res.in}

\author{Mithun Mukherjee}
\address{Mithun Mukherjee, School of Mathematical and Computational Sciences, Indian Association for the Cultivation of Science, 2A and 2B Raja S C Mullick Road, Jadavpur, Kolkata-700032, India.}
\email{mithun.mukherjee@iacs.res.in, mithunmukh@gmail.com}

\begin{document}


\begin{abstract}
	Let $J_n$ be the Jordan block of size $n$ with all eigenvalues zero. Arveson introduced the notion of the matricial range of an operator in his remarkable article called Subalgebras of $C^*$-algebras II (Acta Math, 128, 1972) and established that every unital positive map on the operator system generated by $J_2$ is completely positive. This describes the matricial range of $J_2$ as the set of all matrices with numerical radius at most $\frac{1}{2}$. After fifty years of the above result of Arveson, the matricial range of $J_n$ for $n\geq 3$ has not been characterized. This article aims to investigate this long-standing open problem for $n=3$. First, we prove that every unital contractive map on the operator system generated by $J_3$ is $2$-positive. Next, we show that every unital contractive map on the operator system generated by $J_3$ is completely positive in certain cases. As a consequence, we obtain a partial description of the matricial range of $J_3$. To establish these results, we develop a structure theorem for a dilation of an operator $B$ satisfying $BB^*+B^*B=I$ and then investigate whether every $B\in\mathbb{M}_n$ satisfying $BB^*+B^*B\leq I_n$ admits a dilation $\widetilde{B}$ for which $\widetilde{B}\widetilde{B}^*+\widetilde{B}^*\widetilde{B}=I$. We show that this has an affirmative answer for $n=2$, but fails to hold in general for $n\geq 3$.  Next, we study unital contractive maps on operator systems generated by $4\times 4$ normal matrices, and show that this is equivalent to studying a unital contractive map on the operator system generated by $T=\text{diag}(\lambda,-1,i,-i)$, where $\Re{(\lambda)}\geq 0$. We prove that every unital contractive map on the operator system generated by $T=\text{diag}(1,-1,i,-i)$ is completely positive. We conclude with some examples, remarks on the matricial range of $J_n$ and a few applications on the constrained unitary dilations of $J_2$.
\end{abstract}



\maketitle


\section{Introduction}

Let $\mathcal{B(\mathcal{H})}$ be the algebra of all bounded linear maps on a Hilbert space $\mathcal{H}$. If $\mathcal{H}$ is an $n$-dimensional Hilbert space, we identify $\mathcal{B}(\mathcal{H})$ with $\mathbb{M}_n,$ the space of all $n\times n$ complex matrices. A set $\mathcal{K}\subseteq\mathcal{B}(\mathcal{H})$ is said to be \textit{$C^*$-convex} if $T_1,\hdots,T_k\in\mathcal{K}$ and $V_1,\hdots,V_k\in\mathcal{B}(\mathcal{H})$ with $\sum_j^kV_j^*V_j=I$ then $\sum_{j=1}^{k}V_j^*T_jV_j\in\mathcal{K}$. This generalizes the notion of linear convexity. Given $T\in\mathcal{B}(\mathcal{H})$, the \textit{$C^*$-convex hull generated by $T$}, denoted by $C^*$-conv$\{T\}$, is defined as the smallest $C^*$-convex set containing $T$, that is,
\begin{align*}
	C^*\text{-conv}\{T\}=\left\{\sum_{j=1}^kV_j^*TV_j: V_j\in\mathcal{B}(\mathcal{H})\text{ with }\sum_{j=1}^kV_j^*V_j=I\right\}.
\end{align*}

In 1972, Arveson \cite{WA} introduced the notion of the matricial range of an operator in his seminal article called Subalgebras of $C^*$-algebras II that generalizes the concept of the $C^*$-convex hull generated by an operator. Let $T\in\mathcal{B}(\mathcal{H})$ and $n\in\mathbb{N}$. The $n$th matricial range of $T$, denoted by $W_n(T)$, is defined as
\begin{align*}
	W_n(T)=\left\{\varphi(T):\varphi:\text{span}\{I,T,T^*\}\rightarrow\mathbb{M}_n\text{ is unital completely positive map}\right\}.
\end{align*}

By Choi-Kraus \cite{Cho75} decomposition of the completely positive maps, we observe that $W_n(T)$ coincides with $C^*$-conv$\{T\}$ whenever $T\in\mathbb{M}_n$. It can be checked that $W_n(T)$ is $C^*$-convex and compact set. If $n=1$, $W_n(T)$ coincides with the closure of the classical numerical range. The following properties of the matricial range are immediate:

\begin{itemize}
	\item[(P1)] $W_n(T^*)=\{X^*: X\in W_n(T)\}$.
	\item[(P2)] $W_n(U^*TU)=W_n(T)$, where $U\in\mathcal{B}(\mathcal{H})$ is unitary.\label{P2}
	\item[(P3)] $W_n(\alpha T+\beta I)=\alpha W_n(T)+\beta$, where $\alpha,\beta\in\mathbb{C}$.\label{P3}
\end{itemize}

Arveson \cite{WA} described the matricial range of a normal operator as the $C^*$-convex hull of its spectrum. Let $T\in\mathcal{B}(\mathcal{H})$ be normal and $n\in\mathbb{N}$. Arveson showed that
\begin{align*}
	W_n(T)=\overline{\left\{\sum_j\lambda_jH_j: \lambda_j\in\sigma(T), H_j\geq 0 \text{ for every } j, \text{ and }\sum_j H_j=I\right\}}.
\end{align*} 
He also determined that the matricial ranges of both the unilateral and bilateral shifts coincide with the set of all contractions. The study of this set has attracted considerable attention. Some relevant references include \cite{BS, DF, FM, LPS, LP2}.

The characterization of the matricial range of an operator $T$ has led to significant interest in the study of unital completely positive maps on the operator system $\text{span}\{I,T,T^*\}$. Let $J_n$ be the Jordan block of size $n$ with all eigenvalues $0$. Arveson \cite{WA} proved that every unital positive map from $\text{span}\{I,J_2,J_2^*\}$ to $\mathcal{B}(\mathcal{H})$ is completely positive. This describes the matricial range of $J_2$ as the set of all matrices whose numerical radius is at most $\frac{1}{2}$. The proof of this statement is far from trivial and depends on subtle results of Ando \cite{TA}. A comprehensive survey for this result can be found in \cite{MA}. An operator $T$ is said to be \textit{quadratic} if it satisfies $T^2+bT+cI=0$ for some $b,c\in\mathbb{C}$. Tso-Wu \cite{TW} generalized the above results of Arveson for quadratic operators. Let $T$ be a quadratic operator. Tso-Wu \cite{TW} showed that every unital positive map from $\text{span}\{I,T,T^*\}$ to $\mathcal{B}(\mathcal{H})$ is completely positive. This describes the matricial range of $T$ as the set of all matrices whose numerical ranges are contained in the closure of the numerical range of $T$.

In 2001, Choi-Li \cite{CL} further generalized the above results for $3\times 3$ matrices which have a reducing subspace. Suppose $T\in\mathbb{M}_3$ has a reducing subspace. Choi-Li \cite{CL} proved that every unital positive map from $\text{span}\{I,T,T^*\}$ to $\mathcal{B}(\mathcal{H})$ is completely positive. The proof relies on a nontrivial result on constrained unitary dilation of a contraction: If $T\in\mathcal{B}(\mathcal{H})$ is a contraction with $T+T^*\leq \mu I$ for some $\mu\in\mathbb{R}$ then $T$ has a unitary dilation $U\in\mathcal{B}(\mathcal{H\oplus H})$ satisfying $U+U^*\leq \mu I$. Subsequently, in 2019, Li-Poon \cite{LP} provided an alternative proof of the same result, employing different techniques from the theory of positive and completely positive maps.  An operator system is said to be \textit{maximal} if every positive map from it to another operator system is completely positive. Li-Poon \cite{LP3} characterized the maximal operator system which extends some of the above results.

After fifty years of the above result of Arveson, the matricial range of $J_n$ for $n\geq 3$ has not been characterized. In fact, there is not even a conjecture describing the matricial range of $J_n$ for $n\geq 3$. The present article aims to investigate this long-standing open problem for $n=3$. First, we prove that every unital contractive map on the operator system generated by $J_3$ is $2$-positive. Next, we show that every unital contractive map on the operator system generated by $J_3$ is completely positive in certain cases. As a consequence, we obtain a partial description of the matricial range of $J_3$. To establish these results, we develop a structure theorem for a dilation of an operator $B$ satisfying $BB^*+B^*B=I$ and then investigate whether every $B\in\mathbb{M}_n$ satisfying $BB^*+B^*B\leq I_n$ admits a dilation $\widetilde{B}$ for which $\widetilde{B}\widetilde{B}^*+\widetilde{B}^*\widetilde{B}=I$. We show that this has an affirmative answer for $n=2$, but fails to hold in general for $n\geq 3$.

Next, we study unital contractive maps on operator systems generated by $4\times 4$ normal matrices, and show that this is equivalent to studying a unital contractive map on the operator system generated by $T=\text{diag}(\lambda,-1,i,-i)$, where $\Re{(\lambda)}\geq 0$. We prove that every unital contractive map on the operator system generated by $T=\text{diag}(1,-1,i,-i)$ is completely positive. As a consequence, we deduce that every unital contractive map on the operator system generated by normal matrices obtained by an invertible affine transformation to $T=\text{diag}(1,-1,i,-i)$ is also completely positive. We discuss some remarks on the matricial range of $J_n$ and present several examples of contractive and non-contractive maps on operator systems generated by normal matrices and derive some applications on the constrained unitary dilations of $J_2$. We conclude with an example of a unital contractive map on the operator system generated by a $3\times 3$ matrix which is not completely positive.

The structure of the paper is as follows. In Section \ref{preliminaries}, we provide the necessary preliminaries to set the stage for our main results. Next, we briefly discuss about the matricial range of an operator in Section \ref{matricial range}. In Section \ref{dilation}, we establish a structure theorem for a dilation of an operator $B$ satisfying $BB^*+B^*B=I$ and then investigate whether every $B\in\mathbb{M}_n$ satisfying $BB^*+B^*B\leq I_n$ admits a dilation $\widetilde{B}$ for which $\widetilde{B}\widetilde{B}^*+\widetilde{B}^*\widetilde{B}=I$.  We use this to prove that every unital contractive map on the operator system generated by $J_3$ is $2$-positive in Section \ref{2-positive}. In Section \ref{3-positive}, we show that every unital contractive map on the operator system generated by $J_3$ is completely positive in certain cases. In Section \ref{normal}, we study unital contractive maps on operator systems generated by $4\times 4$ normal matrices, and show that this is equivalent to studying a unital contractive map on the operator system generated by $T=\text{diag}(\lambda,-1,i,-i)$, where $\Re{(\lambda)}\geq 0$. We prove that every unital contractive map on the operator system generated by $T=\text{diag}(1,-1,i,-i)$ is completely positive. We conclude with some examples, remarks on the matricial range of $J_n$ and a few applications on the constrained unitary dilations of $J_2$ in Section \ref{remarks and examples}.

\section{Preliminaries}\label{preliminaries}

In this section, we present essential definitions and preliminary results to set the stage for discussions in the subsequent sections. Let us begin with the following definition.

\begin{defn}[Operator system]
	Let $\mathcal{A}$ be a $C^*$-algebra with unit $1$. A subspace $\mathcal{S}\subseteq\mathcal{A}$ is said to be an \textit{operator system} if $1\in\mathcal{S}$ and $	\mathcal{S}=\mathcal{S}^*:=\{a^*:a\in\mathcal{S}\}$.
\end{defn}

 Let $T\in\mathcal{B(\mathcal{H})}$. Denote $\mathcal{OS}(T):=\text{span}\{I,T,T^*\}$. Note that $\mathcal{OS}(T)$ is an operator system. We call $\mathcal{OS}(T)$ as an \textit{operator system generated by $T$}.

 Let $\mathcal{A, B}$ be two $C^*$-algebras with unit $1$ and $\mathcal{S}\subseteq\mathcal{A}$ be an operator system. Suppose $\varphi:\mathcal{S}\rightarrow\mathcal{B}$ be a linear map. Define $\varphi_n:\mathbb{M}_n(\mathcal{S})\rightarrow \mathbb{M}_n(\mathcal{B})$ such that
 \begin{align*}
 	\varphi_n((a_{ij}))=(\varphi(a_{ij})),\;  (a_{ij})\in\mathbb{M}_n(\mathcal{S}).
 \end{align*}
 Then we define the following:
 \begin{itemize}
 	\item $\varphi$ is \textit{self-adjoint preserving} if $\varphi(a^*)=\varphi(a)^*$ for all $a\in\mathcal{S}$.
 	\item $\varphi$ is \textit{positive} if $\varphi(a)\geq 0$ whenever $a\geq 0$.
 	\item $\varphi$ is \textit{unital} if $\varphi(1)=1$.
 	\item $\varphi$ is $n$-positive if $\varphi_n$ is positive.
 	\item $\varphi$ is \textit{completely positive} (CP) if $\varphi$ is $n$-positive for every $n$.
 	\item $\varphi$ is \textit{unital completely positive} (UCP) if $\varphi$ is unital and completely positive.
 	\item $\varphi$ is \textit{contractive} if $\Vert \varphi(a)\Vert\leq\Vert a\Vert$ for all $a\in\mathcal{S}$.
 \end{itemize}

 One can readily verify the following chain of implications for a unital linear map $\varphi$:
 \begin{align*}
 	\varphi \text{ is CP}\Rightarrow \varphi \text{ is } 2\text{-positive}\Rightarrow\varphi \text{ is contractive}\Rightarrow\varphi \text{ is positive}.
 \end{align*}
 However, none of the converse implications hold in general.

 We recall that a transformation from $\mathbb{C}^n$ to itself is said to be \textit{affine} if it is of the form $x\mapsto Rx+y$ where $R\in\mathbb{M}_n$ and $y\in\mathbb{C}^n$. Such a transformation is invertible if and only if $R$ is invertible. This notion can be extended to $n$-tuples of operators on a Hilbert space $\mathcal{H}$ by defining:
 \begin{align*}
 	(T_1,\hdots,T_n)\mapsto (T_1,\hdots,T_n)R+(s_1 I,\hdots,s_n I)
 \end{align*}
 where $R=(r_{ij}I),\; r_{ij}\in\mathbb{C}$ for all $i,j$ and $s_1,\hdots, s_n\in\mathbb{C}$. The following lemma is immediate.

 \begin{lem}[\cite{LP}]\label{affine transformation}
 	Let $\mathcal{S}\subseteq\mathcal{B}(\mathcal{H})$ be an operator system with a basis $\{I, T_1,\hdots, T_n\}$, and let $X_1,\hdots, X_n\in\mathcal{B}(\mathcal{K})$. Suppose $\varphi:\mathcal{S}\rightarrow \mathcal{B}(\mathcal{K})$ is a unital linear map satisfying $\varphi(T_j)=X_j$ for all $j$. Let $f$ be an invertible affine transformation such that $f(T_1,\hdots, T_n)=(\widetilde{T}_1,\hdots, \widetilde{T}_n)$ and $f(X_1,\hdots, X_n)=(\widetilde{X}_1,\hdots, \widetilde{X}_n)$. Suppose $\widetilde{\varphi}:\mathcal{S}\rightarrow \mathcal{B}(\mathcal{K})$ is a unital linear map satisfying $\widetilde{\varphi}(\widetilde{T}_j)=\widetilde{X}_j$ for all $j$. Then $\widetilde{\varphi}$ is positive (respectively, contractive, $n$-positive, or completely positive) if and only if $\varphi$ is positive (respectively, contractive, $n$-positive, or completely positive).
 \end{lem}

 Let $T\in\mathcal{B(\mathcal{H})}$. The \textit{numerical range} of $T$, denoted by $W(T)$, is defined as
 \begin{align*}
 	W(T)=\{\langle Tf,f\rangle:\Vert f\Vert=1\}.
 \end{align*}
 $W(T)$ is non-empty, convex set. Moreover, if $T$ acts on finite dimension then $W(T)$ is compact. The closure of the numerical range of a normal operator is the closed convex hull of its spectrum. The \textit{numerical radius} of $T$, denoted by $w(T)$, is defined as
 \begin{align*}
 	w(T)=\sup_{\Vert f\Vert=1}\vert\langle Tf,f\rangle\vert.
 \end{align*}

  Let $H\in\mathcal{B}(\mathcal{H})$ be self-adjoint. Define
 \begin{align}
 	\lambda_1(H):=\sup_{\Vert f\Vert=1}\langle Hf,f\rangle.
 \end{align}
 The closure of the numerical range of an operator can be described as the intersection of the closed half-planes. Let $T\in\mathcal{B(\mathcal{H})}$. Then
 \begin{align}\label{half plane description of the numerical range}
 	\overline{W(T)}=\bigcap_{\theta\in[0,2\pi)}\{z\in\mathbb{C}:\Re{(e^{i\theta}z)}\leq\lambda_1(\Re{(e^{i\theta}T))}\}.
 \end{align}
 Using this, it can be shown that the numerical range of the Jordan block of size $n$ with all eigenvalues $0$ is the disc centred at $0$ with radius $\cos\frac{\pi}{n+1}$. As a consequence, it follows that the numerical ranges of both the unilateral and bilateral shifts are equal to the open unit disc. The following lemma provides a useful criterion for determining whether a map on an operator system generated by $T$ is positive in terms of the inclusion of the numerical ranges.

 \begin{lem}\label{equivalent criterion for positivity}
 	Let $T\in\mathcal{B}(\mathcal{H}), X\in\mathcal{B}(\mathcal{K})$. A unital self-adjoint preserving linear map $\varphi:\mathcal{OS}(T)\rightarrow\mathcal{OS}(X)$ with $\varphi(T)=X$ is positive if and only if $W(X)\subseteq\overline{W(T)}$.
\end{lem}

\begin{proof}
	The proof is similar to that of \cite[Lemma 4.1]{CL}.
\end{proof}

The following lemmas are very useful to show a map on an operator system generated by an operator $T$ is $n$-positive. Given $S\subseteq\mathbb{C}, S^\circ$ denote the interior of $S$.

\begin{lem}\label{form of positive element}
	Let $T\in\mathcal{B(\mathcal{H})}$ with $0\in W(T)$. Then every positive element in $\mathbb{M}_n(\mathcal{OS}(T))$ is of the form $A\otimes I+B\otimes T+B^*\otimes T^*$ for some $A\geq 0$ and $B\in\mathbb{M}_n$.
\end{lem}

\begin{proof}
	Let us first assume that $W(T)^\circ=\emptyset$. Indeed, if $W(T)=\{0\}$ then the conclusion follows. Let $W(T)\neq\{0\}$. Then $\mathcal{OS}(T)=\text{span}\{I,T\}$. Suppose $A\otimes I+B\otimes T$ is a positive element of $\mathbb{M}_n(\mathcal{OS}(T))$ where $A, B\in\mathbb{M}_n$. Since $0\in W(T)$, we have $A\geq 0$. This gives $\Im{(B\otimes T)}=0$. Therefore, $B\otimes T$ is self-adjoint. Hence the conclusion follows. Next, let $W(T)^\circ\neq\emptyset$. Suppose $A\otimes I+B\otimes T+C\otimes T^*$ is a positive element of $\mathbb{M}_n(\mathcal{OS}(T))$ where $A, B, C\in\mathbb{M}_n$. Again, since $0\in W(T)$, we have $A\geq 0$. This gives
	\begin{align*}
		\Im{(B\otimes T+C\otimes T^*)}=0\Rightarrow (B-C^*)\otimes T=(B^*-C)\otimes T^*.
	\end{align*}
	So $(B-C^*)\otimes T$ is self-adjoint. We claim that $B=C^*$. If possible let, there exists $a\neq 0$ such that $a\in W(B-C^*)$. Let $\lambda\in W(T)^\circ$. Then there exists $r>0$ such that $\lambda+re^{i\theta}\in W(T)$ for all $\theta\in[0,2\pi)$. So we have $a\lambda\in\mathbb{R}$ and $a(\lambda+re^{i\theta})\in\mathbb{R}$ for all $\theta\in[0,2\pi)$. This is a contradiction. Hence $B=C^*$. This completes the proof.
\end{proof}

\begin{lem}\label{n-positivity}
	Let $T\in\mathcal{B(\mathcal{H})}$ with $0\in W(T)$ and $X\in\mathcal{B}(\mathcal{H})$. A unital self-adjoint preserving linear map $\varphi:\mathcal{OS}(T)\rightarrow\mathcal{OS}(X)$ with $\varphi(T)=X$ is $n$-positive if and only if $I_n\otimes I+B\otimes X+B^*\otimes X^*\geq 0$ whenever $I_n\otimes I+B\otimes T+B^*\otimes T^*\geq 0$ where $B\in \mathbb{M}_n$.
\end{lem}

\begin{proof}
	$``\Leftarrow"$. Let $A\otimes I+B\otimes T+B^*\otimes T^*$ be a positive element in $\mathbb{M}_n(\mathcal{OS}(T))$ where $A\geq 0$ and $B\in\mathbb{M}_n$ (by Lemma \ref{form of positive element}). Suppose $A>0$. Then we have
	\begin{align}\label{strict positivity}
		&A\otimes I+B\otimes T+B^*\otimes T^*\geq 0\nonumber\\
		&\Rightarrow I_n\otimes I+A^{-\frac{1}{2}}BA^{-\frac{1}{2}}\otimes T+(A^{-\frac{1}{2}}BA^{-\frac{1}{2}})^*\otimes T^*\geq 0\nonumber\\
		& \Rightarrow I_n\otimes I+A^{-\frac{1}{2}}BA^{-\frac{1}{2}}\otimes X+(A^{-\frac{1}{2}}BA^{-\frac{1}{2}})^*\otimes X^*\geq 0, 
		\text{ by hypothesis}\nonumber\\
		&\Rightarrow A\otimes I+B\otimes X+B^*\otimes X^*\geq 0.
	\end{align}
	Next, let $A\geq 0$. Suppose $\epsilon>0$ such that $A_\epsilon=A+\epsilon I>0$. So, by Eqn \ref{strict positivity}, we have $A_\epsilon\otimes I+B\otimes X+B^*\otimes X^*\geq 0$. Taking $\epsilon\rightarrow 0$, we get $A\otimes I+B\otimes X+B^*\otimes X^*\geq 0$. Therefore, $\varphi$ is $n$-positive.

	$``\Rightarrow"$. This is straightforward.
\end{proof}

Let $T\in\mathcal{B}(\mathcal{H})$. An operator $S\in\mathcal{B}(\mathcal{K})$ where $\mathcal{K}\supseteq\mathcal{H}$ is said to be a \textit{dilation} of $T$ (or $T$ is said to be a \textit{compression} of $T$) if $T=P_{\mathcal{H}}S|_{\mathcal{H}}$. That is,
\begin{align*}
	S=\begin{pmatrix}
		T & * \\
		* & * \\
	\end{pmatrix}.
\end{align*}
Moreover, if $S$ is unitary then $S$ is called a \textit{unitary dilation} of $T$. Halmos showed that every contraction has a unitary dilation. The following lemma highlights the connection between dilation theory and completely positivity.

\begin{lem}[\cite{CL, VP}]\label{equivalent criterion for CP}
	Let $T\in\mathbb{M}_n$ and $X\in\mathcal{B}(\mathcal{H})$. A unital self-adjoint preserving linear map $\varphi:\mathcal{OS}(T)\rightarrow\mathcal{OS}(X)$ with $\varphi(T)=X$ is completely positive if and only if $X$ is a compression of $T\otimes I$. 
\end{lem}

Let $J_n$ be the Jordan block of size $n$ with all eigenvalues $0$, that is,
\begin{align*}
	J_n=\begin{pmatrix}
		0 & 1 &  &  \\
		& 0 & \ddots &  \\
		&    & \ddots & 1\\
		&  &  & 0\\
	\end{pmatrix}.
\end{align*}
Arveson \cite{WA} proved that every unital positive map from $\mathcal{OS}(J_2)$ to $\mathcal{B}(\mathcal{H})$ is completely positive. Choi-Li \cite{CL} generalized this result of Arveson for $2\times 2$ matrices, and for $3\times 3$ matrices which have a reducing subspace. More precisely, they proved the following.

\begin{thm}[Choi-Li \cite{CL}]\label{Choi-Li}
	Suppose $T\in\mathbb{M}_2$, or $T\in\mathbb{M}_3$ has a reducing subspace. Then every unital positive map from $\mathcal{OS}(T)$ to $\mathcal{B}(\mathcal{H})$ is completely positive.
\end{thm}

 However, this result does not hold in general for $3 \times 3$ matrices that do not have any reducing subspace, or for $4 \times 4$ normal matrices as discussed below.

 \begin{eg}(\cite[Example 2.2]{CL2})\label{positive but not CP for J3_0}
 	Let $X=\begin{pmatrix}
 		0 & \sqrt{2}-\varepsilon \\
 		0 & 0 \\
 	\end{pmatrix}$ where $\varepsilon\in[0,\sqrt{2}-1)$. Consider the unital self-adjoint preserving linear map $\varphi:\mathcal{OS}(J_3)\rightarrow\mathbb{M}_2$ such that $\varphi(J_3)=X$. Since $W(X)\subseteq W(J_3)=\{z\in\mathbb{C}:\vert z\vert\leq\frac{1}{\sqrt{2}}\}$, by Lemma \ref{equivalent criterion for positivity}, $\varphi$ is positive. But $\varphi$ is not completely positive. Because if so, $\sqrt{2}-\varepsilon=\Vert X\Vert=\Vert \varphi(T)\Vert\leq \Vert T\Vert=1$. This is a contradiction.
 \end{eg}
 
 Observe that the matrix $X$ in Example~\ref{positive but not CP for J3_0} is not a contraction. We now provide an example of a contraction $X$ such that the unital positive map $\varphi:\mathcal{OS}(J_3)\rightarrow\mathcal{OS}(X)$ with $\varphi(J_3)=X$ is not completely positive.

\begin{eg}\label{positive but not CP for J3}
	Let $X=\begin{pmatrix}
		0 & 1 \\
		\frac{1}{3} & 0 \\
	\end{pmatrix}$. Consider the unital self-adjoint preserving linear map $\varphi:\mathcal{OS}(J_3)\rightarrow\mathbb{M}_2$ such that $\varphi(J_3)=X$. Compute $W(X)=\{(x,y)\in\mathbb{R}^2: \frac{9x^2}{4}+9y^2=1\}$. Since $W(X)\subseteq W(J_3)=\{z\in\mathbb{C}:\vert z\vert\leq\frac{1}{\sqrt{2}}\}$, by Lemma \ref{equivalent criterion for positivity}, $\varphi$ is positive. But $\varphi$ is not completely positive. Because, take $$B=\begin{pmatrix}
	0 & 0.97\\
	0.22 & 0 \\
	\end{pmatrix}$$ and check that $I_2\otimes I_3+B\otimes J_3+B^*\otimes J_3^*\geq 0$ but $I_2\otimes I_3+B\otimes X+B^*\otimes X^*\ngeq 0$.
\end{eg}

It is noteworthy that the map $\varphi$ in Example \ref{positive but not CP for J3} is also not contractive since $\Vert 0.97X+0.22X^*\Vert=1.04>0.99=\Vert 0.97J_3+0.22J_3^*\Vert$. A natural question that arises at this stage is the following.

\begin{que}\label{question 1}\label{open question 1}
	Is every unital contractive map from $\mathcal{OS}(J_3)$ to $\mathcal{B}(\mathcal{H})$ completely positive?
\end{que}

Next, we present an example that Theorem \ref{Choi-Li} does not hold in general for $4\times 4$ normal matrices. To discuss this, we recall the following result of Choi-Li \cite{CL}. Let $T\in\mathcal{B}(\mathcal{H})$ be self-adjoint and $\mu\in\mathbb{R}$. The notation $T\leq\mu$ means $T\leq \mu I$.

\begin{lem}[\cite{CL}]\label{observation}
	$J_2$ has no normal dilation $N$ satisfying $-\frac{1}{2}\leq\Re(N)\leq\frac{1}{2}$.
\end{lem}

\begin{eg}\label{positive but not CP for 4-by-4 normal}
	Let $T=\text{diag}(\lambda_1,\hdots,\lambda_4)$ where $\lambda_j$'s are four distinct points with $\Re(\lambda_j)=\frac{1}{2}$ for $j=1,2$, $\Re(\lambda_j)=-\frac{1}{2}$ for $j=3,4$ and $\vert\lambda_j\vert=r\geq\frac{1}{2}$ for all $j$. Consider the unital self-adjoint preserving linear map $\varphi:\mathcal{OS}(T)\rightarrow\mathcal{OS}(J_2)$ with $\varphi(T)=J_2$. Indeed, by Lemma \ref{equivalent criterion for positivity}, $\varphi$ is positive since $W(J_2)=\{z\in\mathbb{C}:\vert z\vert\leq\frac{1}{2}\}\subseteq W(T)$. But $\varphi$ is not completely positive. Because if so, by Lemma \ref{equivalent criterion for CP}, $J_2$ has a normal dilation $N=T\otimes I$ satisfying $-\frac{1}{2}\leq\Re(N)\leq\frac{1}{2}$. This contradicts Lemma \ref{observation}. So $\varphi$ is positive but not completely positive.
\end{eg}

\begin{rmk}
	It is noteworthy that the map $\varphi$ in Example \ref{positive but not CP for 4-by-4 normal} is not contractive also. To see this, we compute 
	\begin{align*}
		\Vert I_2+\alpha J_2+\beta J_2^*\Vert^2&=1+\frac{\alpha^2+\beta^2}{2}+\sqrt{\frac{(\alpha^2-\beta^2)^2}{4}+(\alpha+\beta)^2},\\
		\Vert I_4+\alpha T+\beta T^*\Vert^2&=1+\frac{\alpha^2+\beta^2}{4}+\alpha+\beta+\frac{\alpha\beta}{2}+v^2(\alpha-\beta)^2
	\end{align*}
	where $\alpha,\beta>0$ and $v^2=r^2-\frac{1}{4}$. Let $m\in\mathbb{N}$ such that $v^2\leq 10^m$. Now it is sufficient to find $\alpha,\beta>0$ such that 
	\begin{align*}
		\text{LHS}=\alpha+\beta+\frac{\alpha\beta}{2}+10^m(\alpha-\beta)^2<\frac{\alpha^2+\beta^2}{4}+\sqrt{\frac{(\alpha^2-\beta^2)^2}{4}+(\alpha+\beta)^2}=\text{RHS}.
	\end{align*}
	Choose $\alpha=5\times10^{2m}$ and $\beta=(5\times10^m)-2$ and check that LHS$<$RHS. Hence $\varphi$ is not contractive.
\end{rmk}

A natural question that arises at this point is the following:

\begin{que}\label{question 2}
	Let $T\in\mathbb{M}_4$ be normal. Is every unital contractive map from $\mathcal{OS}(T)$ to $\mathcal{B}(\mathcal{H})$ completely positive?
\end{que}

We aim to explore Questions \ref{question 1}, \ref{question 2} in greater details in the subsequent sections.

\section{Matricial Range}\label{matricial range}

In this section, we, being motivated by Eqn \ref{half plane description of the numerical range}, describe the matricial range of an operator as intersection of closed half-planes. Let us start with the description of the matricial range of a self-adjoint operator.

\begin{thm}\label{matricial range of self-adjoint operators}
	Let $T\in\mathcal{B}(\mathcal{H})$ be self-adjoint and $n\in\mathbb{N}$. Then
	\begin{align*}
		W_n(T)=\left\{X\in\mathbb{M}_n: W(X)\subseteq\overline{W(T)}\right\}.
	\end{align*}
\end{thm}

\begin{proof}
	Let $X\in W_n(T)$. Then there exists a UCP map $\varphi:\mathcal{OS}(T)\rightarrow\mathcal{OS}(X)$ such that $X=\varphi(T)$. In particular, $\varphi$ is positive and hence, by Lemma \ref{equivalent criterion for positivity}, $W(X)\subseteq\overline{W(T)}$. Conversely, let $X\in\mathbb{M}_n$ with $W(X)\subseteq\overline{W(T)}$. This implies, by Lemma \ref{equivalent criterion for positivity}, that the unital self-adjoint preserving linear map $\psi:\mathcal{OS}(T)\rightarrow\mathcal{OS}(X)$ with $\psi(T)=X$ is positive. Check that $X$ is self-adjoint as $T$ is self-adjoint. Therefore, $\psi$ is completely positive and hence $X\in W_n(T)$.
\end{proof}

The reader may refer to \cite[Proposition 3.1]{LP3} for an alternative proof of Theorem \ref{matricial range of self-adjoint operators}. Let $T\in\mathcal{B}(\mathcal{H})$. In view of (\hyperref[P3]{P3}), it is evident that describing the matricial range of $T$ is equivalent to describing the matricial range of $\alpha T+\beta I$ for all $\alpha,\beta\in\mathbb{C}$. Therefore, without loss of generality, we may assume that $0\in W(T)$ when describing the matricial range of $T$. The following theorem describes the matricial range of an operator.

\begin{thm}\label{description of the matricial range}
		Let $T\in\mathcal{B}(\mathcal{H})$ with $0\in W(T)$ and $n\in\mathbb{N}$. Then
	\begin{align*}
		W_n(T)=\bigcap\limits_{B\in\mathbb{M}_n}\left\{X\in\mathbb{M}_n: W(B\otimes X)\subseteq\overline{W(B\otimes T)}\right\}.
	\end{align*}
\end{thm}

\begin{proof}

	Let $X\in W_n(T)$. Then there exists a UCP map $\varphi:\mathcal{OS}(T)\rightarrow\mathcal{OS}(X)$ such that $X=\varphi(T)$. Let $B\in\mathbb{M}_n$. Since $\varphi$ is $n$-positive, the unital self-adjoint preserving map $\psi:\mathcal{OS}(B\otimes T)\rightarrow\mathcal{OS}(B\otimes X)$ with $\psi(B\otimes T)=B\otimes X$ is positive and hence, by Lemma \ref{equivalent criterion for positivity}, $W(B\otimes X)\subseteq\overline{W(B\otimes T)}$. Therefore, $X\in\text{RHS}$.

	Conversely, let $X\in$ RHS. Consider the unital self-adjoint preserving linear map $\xi:\mathcal{OS}(T)\rightarrow\mathcal{OS}(X)$ such that $\xi(T)=X$. Let $B\in\mathbb{M}_n$ such that $I_n\otimes I+B\otimes T+B^*\otimes T^*\geq 0$. Since $W(B\otimes X)\subseteq\overline{W(B\otimes T)}$, by Lemma \ref{equivalent criterion for positivity}, the unital self-adjoint preserving linear map $\eta:\mathcal{OS}(B\otimes T)\rightarrow\mathcal{OS}(B\otimes X)$ with $\eta(B\otimes T)=B\otimes X$ is positive and hence $I_n\otimes I_n+B\otimes X+B^*\otimes X^*\geq 0$. This implies, by Lemma \ref{n-positivity}, that $\xi$ is $n$-positive for every $n$ and hence completely positive. Therefore, $X\in W_n(T)$. This completes the proof.
\end{proof}

As a consequence, we obtain a description of the matricial range of an operator as intersection of closed half-planes discussed below.

\begin{cor}\label{half-planes description of the matricial range}
	Let $T\in\mathcal{B}(\mathcal{H})$ with $0\in W(T)$ and $n\in\mathbb{N}$. Then
	\begin{align*}
		W_n(T)=\bigcap\limits_{B\in\mathbb{M}_n}\left\{X\in\mathbb{M}_n: \Re{(B\otimes X)}\leq\lambda_1(\Re{(B\otimes T)})\right\}.
	\end{align*}
\end{cor}

\begin{proof}
	Let $X\in W_n(T)$. Suppose $B\in\mathbb{M}_n$. Then, by Theorem \ref{description of the matricial range}, $W(B\otimes X)\subseteq \overline{W(B\otimes T)}$. Let $f\in\mathbb{C}^n\otimes\mathbb{C}^n$ with $\Vert f\Vert=1$. Then
	\begin{align*}
		\langle \Re{(B\otimes X)} f,f\rangle
		&=\Re\langle B\otimes Xf,f\rangle\\
		&=\lim_{n\rightarrow\infty}\Re\langle B\otimes T g_n,g_n\rangle, \text{ for some } g_n\in\mathbb{C}^n\otimes\mathcal{H} \text{ with } \Vert g_n\Vert=1\\
		&=\lim_{n\rightarrow\infty}\langle \Re{(B\otimes T)} g_n,g_n\rangle\\
		&\leq\lambda_1(\Re{(B\otimes T)}).
	\end{align*}
	Therefore, $X\in\text{RHS}$. Conversely, let $X\in\text{RHS}$. Suppose $B\in\mathbb{M}_n$. By Theorem \ref{description of the matricial range}, it is sufficient to prove $W(B\otimes X)\subseteq \overline{W(B\otimes T)}$. Let $z\notin\overline{W(B\otimes T)}$. Then, by Eqn \ref{half plane description of the numerical range}, there exists $\theta\in[0,2\pi)$ such that $\Re{(e^{i\theta}z)}>\lambda_1(\Re{(e^{i\theta}B\otimes T)})\geq\Re{(e^{i\theta}B\otimes X)}$. This implies that $\Re{(e^{i\theta}z)}>\lambda_1(\Re{(e^{i\theta}B\otimes X)})$. So, by Eqn \ref{half plane description of the numerical range}, $z\notin\overline{W(B\otimes X)}$ and hence $z\notin W(B\otimes X)$. This completes the proof.
\end{proof}

\section{Dilation of an operator $B$ satisfying $BB^*+B^*B\leq I$}\label{dilation}

In this section, we establish a structure theorem for a dilation of an operator $B$ satisfying $BB^*+B^*B=I$ and then investigate whether every $B\in\mathbb{M}_n$ satisfying $BB^*+B^*B\leq I_n$ admits a dilation $\widetilde{B}$ for which $\widetilde{B}\widetilde{B}^*+\widetilde{B}^*\widetilde{B}=I$. This analysis plays the central role in the subsequent sections. Beyond this specific application, the results are significant in the dilation theory also. The following theorem provides the structure of a dilation of $B\in\mathbb{M}_n$ satisfying $BB^*+B^*B=I$.

\begin{thm}\label{Structure of a dilation of B satisfying BB*+B*B=I}
	Let $B\in\mathbb{M}_n$ be such that $BB^*+B^*B=I_n$. Then $B$ has a dilation $\widetilde{B}$ of the form
	\begin{align*}
		\widetilde{B}=\bigoplus_{\substack{\alpha,\beta\in\mathbb{C} \\ \vert\alpha\vert^2+\vert\beta\vert^2=1}} \begin{pmatrix}
			0 & \alpha \\
			\beta & 0 \\
		\end{pmatrix}
	\end{align*}
\end{thm}

To prove this, we need the following lemma.

\begin{lem}\label{lemma for structure theorem}
	Let $B\in\mathbb{M}_n$ satisfying $BB^*+B^*B=I_n$. Then there exists a unitary $V\in\mathbb{M}_n$ such that 
	\begin{align*}
		V^*\vert B\vert V=\bigoplus_{0\leq \lambda\leq 1}\begin{pmatrix}
			\lambda I_r & 0 \\
			0 & \sqrt{1-\lambda^2} I_s \\
		\end{pmatrix},\quad
		V^*\vert B^*\vert V=\bigoplus_{0\leq \lambda\leq 1}\begin{pmatrix}
			\sqrt{1-\lambda^2}I_r & 0 \\
			0 & \lambda I_s \\
		\end{pmatrix}
	\end{align*}
	where $\lambda$'s are distinct for different summands in the decomposition and $r=r(\lambda), s=s(\lambda)\in\mathbb{N}$.
\end{lem}

\begin{proof}
	We note that $\vert B\vert$ and $\vert B^*\vert$ have the same set of eigenvalues with equal multiplicity. Since $\vert B\vert^2+\vert B^*\vert^2=I_n$, there exists a unitary $U\in\mathbb{M}_n$ such that $U^*\vert B\vert U=\text{diag}(\lambda_1,\hdots,\lambda_n)$ and $U^*\vert B^*\vert U=\text{diag}(\lambda_{\sigma(1)},\hdots, \lambda_{\sigma(n)})$ where $0\leq\lambda_j\leq 1$ for all $j$ and $\sigma$ is a permutation on $\{1,\hdots,n\}$. Let $\sigma=\sigma_1\hdots\sigma_m$ where $\sigma_1,\hdots,\sigma_m$ are disjoint cycles. Suppose $\sigma_j=(i_1,\hdots,i_k)$ where $1\leq j\leq m$ and $1\leq k\leq n$. Check that $\lambda_{i_d}^2+\lambda_{i_{d+1}}^2=1$ for all $1\leq d\leq k$ where $\lambda_{i_{k+1}}=\lambda_{i_1}$. If $k$ is odd, this forces $\lambda_{i_d}=\frac{1}{\sqrt{2}}$ for all $1\leq d\leq p$. Let $k=2t$ for some $t$. Then $\lambda_{i_1}=\lambda_{i_3}=\cdots=\lambda_{i_{(2t-1)}}$ and $\lambda_{i_2}=\lambda_{i_4}=\cdots=\lambda_{i_{2t}}$. Therefore, there exists a unitary $V\in\mathbb{M}_n$ such that 
	\begin{align*}
		V^*\vert B\vert V&=\bigoplus_{0\leq \lambda\leq 1}\begin{pmatrix}
			\lambda I_r & 0 \\
			0 & \sqrt{1-\lambda^2} I_s \\
		\end{pmatrix},\\
		V^*\vert B^*\vert V&=\bigoplus_{0\leq \lambda\leq 1}\begin{pmatrix}
			\sqrt{1-\lambda^2}I_r & 0 \\
			0 & \lambda I_s \\
		\end{pmatrix}
	\end{align*}
where $r=r(\lambda), s=s(\lambda)\in\mathbb{N}$. The remaining part of the proof that $\lambda$'s are distinct for different summands in the decomposition follows by applying unitary similarity.
\end{proof}

\begin{proof}[\textbf{Proof of Theorem \ref{Structure of a dilation of B satisfying BB*+B*B=I}}]
	We consider 
	\begin{align*}
		\widetilde{B}=\begin{pmatrix}
			0 & B \\
			B & 0 \\
		\end{pmatrix}.
	\end{align*}
	Let $f,g\in\mathbb{C}^n$ satisfying $\Vert f\Vert=\Vert g\Vert=1$. Note that $\langle\widetilde{B}\frac{f\oplus f}{\sqrt{2}},\frac{g\oplus g}{\sqrt{2}}\rangle=\langle Bf,g\rangle$. Therefore, $\widetilde{B}$ is a dilation of $B$. Now, by polar decomposition, we write $B=U\vert B\vert$ where $U\in\mathbb{M}_n$ is unitary. Observe that $	\vert B^*\vert=(BB^*)^{\frac{1}{2}}=(U\vert B\vert^2U^*)^{\frac{1}{2}}=U\vert B\vert U^*$. Using this, we obtain
	\begin{align*}
		\begin{pmatrix}
			U^* &  \\
			    & I \\
		\end{pmatrix}\begin{pmatrix}
		0 & B \\
		B & 0 \\
		\end{pmatrix}\begin{pmatrix}
		U & 0 \\
		0 & I \\
		\end{pmatrix}=\begin{pmatrix}
		0 & \vert B\vert \\
		\vert B^*\vert U^2 & 0 \\
		\end{pmatrix}.
	\end{align*}

	Next, by Lemma \ref{lemma for structure theorem}, there exists a unitary $V\in\mathbb{M}_n$ such that $V^*\vert B\vert V=\oplus_{\lambda} B_\lambda$ and $V^*\vert B^*\vert V=\oplus_{\lambda}B_\lambda'$ where $0\leq\lambda\leq 1$,
		\begin{align*}
		B_\lambda=\begin{pmatrix}
			\lambda I_r & 0 \\
			0 & \sqrt{1-\lambda^2} I_s \\
		\end{pmatrix},\quad
		B_\lambda'=\begin{pmatrix}
			\sqrt{1-\lambda^2}I_r & 0 \\
			0 & \lambda I_s \\
		\end{pmatrix}
	\end{align*}
	such that $\lambda$'s are distinct for different summands in the decomposition and $r=r(\lambda), s=s(\lambda)\in\mathbb{N}$. Let $V^*\vert B\vert V$ and $V^*\vert B^*\vert V$ be written with respect to the orthonormal basis $\{e_j\}_{j=1}^\infty$. Now observe that $\langle U\vert B\vert e_j,e_i\rangle=\langle \vert B^*\vert U e_j,e_i\rangle
	\Rightarrow (\lambda_j-\lambda_{\sigma(i)})\langle Ue_j,e_i\rangle=0$ where $\vert B\vert e_j=\lambda_j e_j, \vert B^*\vert e_j=\lambda_{\sigma(j)} e_j$ for all $1\leq j\leq n$ and $\sigma$ is a permutation on $\{1,\cdots,n\}$. This forces that $V^*UV=\oplus_\lambda U_\lambda$ where 
	\begin{align}\label{1'}
		U_\lambda=\begin{pmatrix}
			0 & U_{12}(\lambda) \\
			U_{21}(\lambda) & 0 \\
		\end{pmatrix}
	\end{align}
	corresponding to $\lambda\neq\frac{1}{\sqrt{2}}$. Therefore, $B$ has a dilation of the form
	\begin{align*}
		\bigoplus_\lambda \begin{pmatrix}
			0 & B_\lambda \\
			B_\lambda'U_\lambda^2 & 0 \\
		\end{pmatrix}.
	\end{align*}

	Now we argue that every summand in the previous decomposition of the dilation of $B$ decomposes to the desired form. Indeed, if $\lambda=\frac{1}{\sqrt{2}}$ then $B_\lambda=B_\lambda'=\frac{1}{\sqrt{2}}I_{r+s}$ and hence the conclusion follows by applying unitary similarity. Also, if $\lambda\neq\frac{1}{\sqrt{2}}$ then the conclusion follows using the form of $U_\lambda$ given in Eqn \ref{1'} and applying unitary similarity. This completes the proof.
\end{proof}

Next, we generalize Theorem \ref{Structure of a dilation of B satisfying BB*+B*B=I} for an operator. Most precisely, we have the following.

\begin{thm}\label{Structure theorem for infinite version}
	Let $B\in\mathcal{B}(\mathcal{H})$ be such that $BB^*+B^*B=I$. Then $B$ has a dilation of the form
	\begin{align*}
		\int\limits_{\vert\alpha\vert^2+\vert\beta\vert^2=1}\begin{pmatrix}
			0 & \alpha \\
			\beta & 0 \\
		\end{pmatrix}dE(\alpha,\beta)
	\end{align*}
	where $E$ is a spectral measure on the Borel $\sigma$-algebra of $\mathbb{C}^2$ (relative to $\mathcal{H}$). 
\end{thm}

To prove this, we need the following lemma.

\begin{lem}\label{lemma for infinite structure theorem}
	Let $B\in\mathcal{B}(\mathcal{H})$ be such that $BB^*+B^*B=I$. Then there exists a unitary $U\in\mathcal{B}(\mathcal{H})$ such that $B=U\vert B\vert$.
\end{lem}

\begin{proof}
	By polar decomposition, we write $B=V\vert B\vert$ where $V\in\mathcal{B}(\mathcal{H})$ is a partial isometry with $\text{ker}V=\text{ker}\vert B\vert$. That is, $V$ is an isometry from $\overline{\text{ran}\vert B\vert}$ to $\overline{\text{ran}B}$. Also, $V$ is surjective from $\overline{\text{ran}\vert B\vert}$ to $\overline{\text{ran}B}$.  Therefore, $V$ is unitary from $\overline{\text{ran}\vert B\vert}$ to $\overline{\text{ran}B}$. Now we claim that $\text{dim}(\overline{\text{ran}\vert B\vert}^\perp)=\text{dim}(\overline{\text{ran}B}^\perp)$, that is, $\text{dim}(\text{ker}\vert B\vert)=\text{dim}(\text{ker}B^*)$. Define $T:\text{ker}\vert B\vert\rightarrow\text{ker}B^*$ such that $Tf=B^*f$ where $f\in\text{ker}\vert B\vert$. Indeed, $T$ is well-defined. To see this, given $f\in\text{ker}\vert B\vert$, we compute
	\begin{align*}
		\Vert {B^*}^2f\Vert^2=\langle B{B^*}^2f,B^*f\rangle=\langle (I-B^*B)B^*f,B^*f\rangle=\Vert B^*f\Vert^2-\Vert BB^*f\Vert^2=\Vert f\Vert^2-\Vert f\Vert^2=0
	\end{align*}
	since $\text{ker}B=\text{ker}B^*B=\text{ker}\vert B\vert$. So $T$ is well-defined. Let $f\in\text{ker}\vert B\vert$ with $B^*f=0$. Then $\Vert f\Vert^2=\Vert Bf\Vert^2+\Vert B^*f\Vert^2=0$. Therefore, $f=0$. This implies that $T$ is injective. Hence $\text{dim}(\text{ker} B)=\text{dim}(\text{ker}\vert B\vert)\leq\text{dim}(\text{ker}B^*)$. By symmetry, we also have $\text{dim}(\text{ker}B^*)\leq \text{dim}(\text{ker} B)=\text{dim}(\text{ker}\vert B\vert)$. Therefore, $\text{dim}(\text{ker}\vert B\vert)=\text{dim}(\text{ker}B^*)$. This implies that the unitary $V|_{\overline{\text{ran}\vert B\vert}}$ can be extended to a unitary on $\mathcal{H}$, say, $U$. Also, check that $B=U\vert B\vert$. This completes the proof.
\end{proof}

Let $T\in\mathcal{B}(\mathcal{H})$ be normal. Denote $E_T$, the spectral measure corresponding to $T$.\\

\noindent\textbf{\textit{Proof of Theorem \ref{Structure theorem for infinite version}}}
	We consider 
	\begin{align*}
		\widetilde{B}=\begin{pmatrix}
			0 & B \\
			B & 0 \\
		\end{pmatrix}.
	\end{align*}
	Let $u,v\in\mathcal{H}$ with $\Vert u\Vert=\Vert v\Vert=1$. Note that $\langle\widetilde{B}\frac{u\oplus u}{\sqrt{2}},\frac{v\oplus v}{\sqrt{2}}\rangle=\langle Bu,v\rangle$. Therefore, $\widetilde{B}$ is a dilation of $B$. Now, by Lemma \ref{lemma for infinite structure theorem}, we write $B=U\vert B\vert$ where $U\in\mathcal{B}(\mathcal{H})$ is unitary. Observe that $\vert B^*\vert=(BB^*)^{\frac{1}{2}}=(U\vert B\vert^2U^*)^{\frac{1}{2}}=U\vert B\vert U^*$. Using this, we obtain
		\begin{align}\label{eqn-infinite structure}
		\begin{pmatrix}
			U&  \\
			& I \\
		\end{pmatrix}\begin{pmatrix}
			0 & B \\
			B & 0 \\
		\end{pmatrix}\begin{pmatrix}
			U^* & 0 \\
			0 & I \\
		\end{pmatrix}=\begin{pmatrix}
			0 & U^2\vert B\vert \\
			\vert B^*\vert & 0 \\
		\end{pmatrix}.
	\end{align}

	Now we prove that $U^2$ and $\vert B\vert$ commute. Since $U\vert B\vert=\vert B^*\vert U$, we have $U\vert B\vert^k=\vert B^*\vert^kU$ for all $k\geq 1$. Therefore, $Up(\vert B\vert)=p(\vert B^*\vert)U$ for all polynomial $p$. This implies that $Uf(\vert B\vert)=f(\vert B^*\vert)U$ for all continuous function $f$ supported on $\sigma(\vert B\vert)$. This gives that $Ug(\vert B\vert)=g(\vert B^*\vert)U$ for all bounded Borel measurable functions $g$ supported on $\sigma(\vert B\vert)$. In particular, we have $U\chi_{\omega}(\vert B\vert)=\chi_{\omega}(\vert B^*\vert)U$ where $\chi_{\omega}$ is a characteristic function on a Borel set $\omega$. This gives that $UE_{\vert B\vert}(\omega)=E_{\vert B\vert}(\sqrt{1-\omega^2})U$ where $\sqrt{1-\omega^2}=\{\sqrt{1-z^2}:z\in\omega\}$. Therefore, $U^2E_{\vert B\vert}(\omega)=E_{\vert B\vert}(\omega)U^2$. Hence $U^2\vert B\vert=\vert B\vert U^2$. This implies, by \cite[Theorem 4.10]{SK}, that $E=E_{\vert B\vert}\times E_{U^2}$, the product of the spectral measures of $E_{\vert B\vert}$ and $E_{U^2}$,  is also a spectral measure  on the Borel $\sigma$-algebra of $\mathbb{C}^2$ (relative to $\mathcal{H}$) supported on $\sigma(\vert B\vert)\times\sigma(U^2)$. Furthermore, we have
	\small
	\begin{align*}
		\vert B\vert=\int\limits_{\mathbb{C}^2}\lambda dE(\lambda,\mu),\quad \vert B^*\vert=\int\limits_{\mathbb{C}^2}\sqrt{1-\lambda^2} dE(\lambda,\mu),\quad U^2=\int\limits_{\mathbb{C}^2}\mu dE(\lambda,\mu),\quad U^2\vert B\vert=\int\limits_{\mathbb{C}^2}\lambda\mu dE(\lambda,\mu).
	\end{align*}
	\normalsize
	
	Hence, by Eqn \ref{eqn-infinite structure}, we conclude that $B$ has a dilation of the form
	\begin{align*}
		\int\limits_{\vert\alpha\vert^2+\vert\beta\vert^2=1}\begin{pmatrix}
			0 & \alpha \\
			\beta & 0 \\
		\end{pmatrix}dE(\alpha,\beta).\tag*{\qed}
	\end{align*}

\begin{rmk}
	It should be noted that the proof of Theorem \ref{Structure theorem for infinite version} remains valid whenever $\mathcal{H}$ is finite-dimensional. However, the proof of Theorem \ref{Structure of a dilation of B satisfying BB*+B*B=I} provides an alternative approach of the same.
\end{rmk}

An intriguing question that naturally arises at this stage is the following:

\begin{que}\label{dilation question}
	Does every $B\in\mathbb{M}_n$ satisfying $BB^*+B^*B\leq I_n$ admit a dilation $\widetilde{B}$ such that $\widetilde{B}\widetilde{B}^*+\widetilde{B}^*\widetilde{B}=I$?
\end{que}

We now investigate Question \ref{dilation question}. This has an affirmative answer for $n=2$, but fails to hold in general for $n\geq 3$. The following theorem establishes that this question has an affirmative answer for $n=2$.

 \begin{thm}\label{dilation for M2}
 	Let $B\in\mathbb{M}_2$ be such that $BB^*+B^*B\leq I_2$. Then $B$ has a dilation $\widetilde{B}\in\mathbb{M}_4$ such that $\widetilde{B}\widetilde{B}^*+\widetilde{B}^*\widetilde{B}=I_4$.
 \end{thm}

 To prove this, we need the following lemmas.

 \begin{lem}\label{lemma1 for dilation for M2}
 	Let $B\in\mathbb{M}_2$. Then $B$ is unitarily similar to $\begin{pmatrix}
 		\alpha & \beta \\
 		-\overline{\beta} & \delta \\
 	\end{pmatrix}$ where $\alpha,\beta,\delta\in\mathbb{C}$.
 \end{lem}

 \begin{proof}
 	Let $\{e_1,e_2\}$ be an orthonormal basis of $\mathbb{C}^2$ with respect to which $B+B^*$ is diagonal. We write $B$ with respect to this basis as $B= \begin{pmatrix}
 		\alpha & \beta \\
 		\gamma & \delta \\
 	\end{pmatrix}$. Then
 	\begin{align*}
 		B+B^*=\begin{pmatrix}
 			\alpha+\overline{\alpha} & \beta+\overline{\gamma} \\
 		   \gamma+ \overline{\beta} & \delta+\overline{\delta} \\
 		\end{pmatrix}.
 	\end{align*}
 	So $\gamma+\overline{\beta}=0$. Therefore, $\gamma=-\overline{\beta}$. This completes the proof.
  \end{proof}

  \begin{lem}\label{lemma2 for dilation for M2}
  	Let $B=\begin{pmatrix}
  		\alpha & \beta \\
  		-\overline{\beta} & \delta \\
  	\end{pmatrix}\in\mathbb{M}_2$ be such that $BB^*+B^*B\leq I_2$ where $\alpha,\beta,\delta\in\mathbb{C}$. Then $B$ has a dilation $\widetilde{B}\in\mathbb{M}_4$ such that $\widetilde{B}\widetilde{B}^*+\widetilde{B}^*\widetilde{B}=I_4$.
  \end{lem}

  \begin{proof}
  	We consider
  	\begin{align*}
  		\widetilde{B}=\begin{pmatrix}
  			\alpha & \beta & x_1 & 0\\
  			-\overline{\beta} & \delta & 0 & x_2\\
  			 \overline{x_1} & 0 & r_1 & 0\\
  			 0 & \overline{x_2} & 0 & r_2\\
  		\end{pmatrix}
  	\end{align*}
  	where $x_j,r_j\in\mathbb{C}$ for all $j$ such that $\vert x_1\vert=\sqrt{\frac{1}{2}-\vert\alpha\vert^2-\vert\beta\vert^2},\; \vert x_2\vert=\sqrt{\frac{1}{2}-\vert\beta\vert^2-\vert\delta\vert^2},\; \Re{(r_1)}=-\Re{(\alpha)},\; \vert r_1\vert=\sqrt{\vert\alpha\vert^2+\vert\beta\vert^2},\; \Re{(r_2)}=-\Re{(\delta)}, \text{ and }\vert r_2\vert=\sqrt{\vert\beta\vert^2+\vert\delta\vert^2}$. Indeed, $\widetilde{B}$ is a dilation of $B$. Check that $\widetilde{B}\widetilde{B}^*+\widetilde{B}^*\widetilde{B}=I_4$. Therefore, $\widetilde{B}\in\mathbb{M}_4$ is a dilation of $B$ satisfying $\widetilde{B}\widetilde{B}^*+\widetilde{B}^*\widetilde{B}=I_4$. This completes the proof.
  \end{proof}

  \begin{proof}[\textbf{Proof of Theorem \ref{dilation for M2}}]
  	The proof follows from Lemma \ref{lemma1 for dilation for M2} and \ref{lemma2 for dilation for M2}.
  \end{proof}

We now present two examples demonstrating that Question \ref{dilation question} has a negative answer in general for $n\geq 3$. The first example involves a weighted permutation matrix, whereas the second is of a different nature.

\begin{eg}\label{counter example 1}
	We consider
	\begin{align*}
		B=\begin{pmatrix}
			0 & 0 & \frac{1}{\sqrt{3}}\\
			\sqrt{\frac{2}{3}} & 0 & 0\\
			       0 & \frac{1}{\sqrt{3}} & 0\\
		\end{pmatrix}.
	\end{align*}
	Then $BB^*+B^*B=\text{diag}(1,1,\frac{2}{3})$. We claim that $B$ does not have any dilation $\widetilde{B}$ such that $\widetilde{B}\widetilde{B}^*+\widetilde{B}^*\widetilde{B}=I$. If possible let, $B$ has a dilation $\widetilde{B}$ acting on $\mathcal{H}\supseteq\mathbb{C}^3$ such that $\widetilde{B}\widetilde{B}^*+\widetilde{B}^*\widetilde{B}=I$. Then $\widetilde{B}$ must be of the following form
	\begin{align*}
		\widetilde{B}=\begin{pmatrix}
			0 & 0 & \frac{1}{\sqrt{3}} & 0\\
			\sqrt{\frac{2}{3}} & 0 & 0 & 0\\
			0 & \frac{1}{\sqrt{3}} & 0 & \alpha^*\\
			0 & 0 & \beta & \delta\\
		\end{pmatrix}
	\end{align*} 
	where $\alpha,\beta\in\mathcal{K}=\mathcal{H}\ominus\mathbb{C}^{3}$ and $\delta\in\mathcal{B}(\mathcal{K})$. Now we compute
	\begin{align}\label{eq_counter example1}
		\widetilde{B}\widetilde{B}^*+\widetilde{B}^*\widetilde{B}=\begin{pmatrix}
			1 & 0 & 0 & \frac{\beta^*}{\sqrt{3}}\\
			0 & 1 & 0 & \frac{\alpha^*}{\sqrt{3}}\\
			0 & 0 & \frac{2}{3}+\Vert\alpha\Vert^2+\Vert\beta\Vert^2 & \alpha^*\delta^*+\beta^*\delta\\
			\frac{\beta}{\sqrt{3}} & \frac{\alpha}{\sqrt{3}} & \delta\alpha+\delta^*\beta& \alpha\alpha^*+\beta\beta^*+\delta\delta^*+\delta^*\delta\\
		\end{pmatrix}=I.
	\end{align}
	This gives $\alpha=\beta=0$, and hence equating $(3,3)$-entry of Eqn \ref{eq_counter example1}, we get $\frac{2}{3}=1$. This is a contradiction. Therefore, $B$ does not have a dilation $\widetilde{B}$ such that $\widetilde{B}\widetilde{B}^*+\widetilde{B}^*\widetilde{B}=I$. 
\end{eg}

\begin{eg}\label{counter example 2}
	Let us consider
	\begin{align*}
		B=\begin{pmatrix}
			\frac{i}{\sqrt{3}} & 0 & \frac{1}{\sqrt{6}}\\
			 0 & \frac{1}{\sqrt{3}} & \frac{i}{\sqrt{6}}\\
			\frac{1}{\sqrt{6}} & \frac{i}{\sqrt{6}} & 0\\
		\end{pmatrix}.
	\end{align*}
	Then $BB^*+B^*B=\text{diag}(1,1,\frac{2}{3})$. We claim that $B$ does not have any dilation $\widetilde{B}$ such that $\widetilde{B}\widetilde{B}^*+\widetilde{B}^*\widetilde{B}=I$. If possible let, $B$ has a dilation $\widetilde{B}$ acting on $\mathcal{H}\supseteq\mathbb{C}^3$ such that $\widetilde{B}\widetilde{B}^*+\widetilde{B}^*\widetilde{B}=I$. Then $\widetilde{B}$ must be of the following form
	\begin{align*}
		\widetilde{B}=\begin{pmatrix}
			\frac{i}{\sqrt{3}} & 0 & \frac{1}{\sqrt{6}} & 0\\
			0 & \frac{1}{\sqrt{3}} & \frac{i}{\sqrt{6}} & 0\\
			\frac{1}{\sqrt{6}} & \frac{i}{\sqrt{6}} & 0 & \alpha^*\\
			0 & 0 & \beta & \delta\\
		\end{pmatrix}.
	\end{align*} 
	where $\alpha,\beta\in\mathcal{K}=\mathcal{H}\ominus\mathbb{C}^{3}$ and $\delta\in\mathcal{B}(\mathcal{K})$. Now we compute
	\begin{align}\label{eq_counter example2}
		\widetilde{B}\widetilde{B}^*+\widetilde{B}^*\widetilde{B}=\begin{pmatrix}
			1 & 0 & 0 & \frac{(\beta+\alpha)^*}{\sqrt{6}}\\
			0 & 1 & 0 & \frac{i(\beta-\alpha)^*}{\sqrt{6}}\\
			0 & 0 & \frac{2}{3}+\Vert\alpha\Vert^2+\Vert\beta\Vert^2 & \alpha^*\delta^*+\beta^*\delta\\
			\frac{\beta+\alpha}{\sqrt{6}} & -\frac{i(\beta-\alpha)}{\sqrt{6}} & \delta\alpha+\delta^*\beta& \alpha\alpha^*+\beta\beta^*+\delta\delta^*+\delta^*\delta\\
		\end{pmatrix}=I.
	\end{align}
		This gives $\alpha=\beta=0$, and hence equating $(3,3)$-entry of Eqn \ref{eq_counter example2}, we get $\frac{2}{3}=1$. This is a contradiction. Therefore, $B$ does not have a dilation $\widetilde{B}$ such that $\widetilde{B}\widetilde{B}^*+\widetilde{B}^*\widetilde{B}=I$.
\end{eg}

The examples given in Example \ref{counter example 1}, \ref{counter example 2} are extreme points of the convex set $\{B\in\mathbb{M}_3: BB^*+B^*B\leq I_3\}$. The following theorem characterizes the extreme points of the convex set $\mathcal{S}=\{B\in\mathbb{M}_n: BB^*+B^*B\leq I_n\}$. Denote $(\mathbb{M}_n)_{\text{s.a.}}$, the set of all self-adjoint matrices in $\mathbb{M}_n$ and $\mathcal{E}(\mathcal{S})$, the set of all extreme points of the convex set $\mathcal{S}$.  Let $B\in\mathbb{M}_n$. Take $B=B_1+iB_2$ where $B_1,B_2\in(\mathbb{M}_n)_{\text{s.a.}}$. Check that $BB^*+B^*B=2(B_1^2+B_2^2)$. Therefore, characterizing the extreme points of $\mathcal{S}$ is equivalent to characterizing the extreme points of the set $\left\{(B_1, B_2): B_1, B_2\in(\mathbb{M}_n)_{\text{s.a.}}, B_1^2+B_2^2\leq I_n\right\}$.

\begin{thm}\label{extreme point}
	Let $\mathcal{S}=\left\{(B_1, B_2): B_1, B_2\in(\mathbb{M}_n)_{\text{s.a.}}, B_1^2+B_2^2\leq I_n\right\}$. Then, upto unitary similarity, the extreme points of $\mathcal{S}$ are as follows:

	\resizebox{.97\linewidth}{!}{
		\begin{minipage}{\linewidth}
			\begin{align*}
				\mathcal{E}(S)&=\left\{(B_1,B_2): B_1=\begin{pmatrix}
					* & X_1 \\
					* & * \\
				\end{pmatrix},\; B_2=\begin{pmatrix}
				* & X_2 \\
				* & * \\
				\end{pmatrix}\in(\mathbb{M}_n)_{\text{s.a.}},\; B_1^2+B_2^2=\begin{pmatrix}
				I_{n-k} &  0\\
				0   & R_k \\
				\end{pmatrix},\;  \text{rank}(X_1|X_2)=2k,\; 0<R_k<I_k, \; 0\leq k\leq n\right\}.
			\end{align*}
		\end{minipage}
	}
\end{thm}

To prove this, we need a few lemmas.

\begin{lem}\label{lemma1 for extreme point}
	Let $0\leq t\leq 1$. Suppose $B_1=tC_1+(1-t)D_1, B_2=tC_2+(1-t)D_2$ where $B_j,C_j,D_j\in(\mathbb{M}_n)_{\text{s.a.}}$ for all $j$ such that $C_1^2+C_2^2\leq I_n, D_1^2+D_2^2\leq I_n$. Let $v\in\mathbb{C}^n$ be a unit vector such that $\langle (B_1^2+B_2^2)v,v\rangle=1$. Then $B_jv=C_jv=D_jv$ for $j=1,2$.
\end{lem}

\begin{proof}
	We note that $B_1^2+B_2^2=t^2(C_1^2+C_2^2)+(1-t)^2(D_1^2+D_2^2)+t(1-t)\sum_{j=1}^2(C_jD_j+D_jC_j)$. Also, observe that $-2I_n\leq-\sum_{j=1}^2(C_j^2+D_j^2)\leq\sum_{j=1}^2(C_jD_j+D_jC_j)\leq\sum_{j=1}^2(C_j^2+D_j^2)\leq 2I_n$. Since $\langle (B_1^2+B_2^2)v,v\rangle=1$, we have $$\langle (C_1^2+C_2^2)v,v\rangle=1,\; \langle (D_1^2+D_2^2)v,v\rangle=1,\text{ and }\sum_{j=1}^2\langle (C_jD_j+D_jC_j)v,v\rangle=2.$$ Therefore, we obtain
	\begin{align*}
		&\sum_{j=1}^2\left\langle (C_jD_j+D_jC_j)v,v\right\rangle=\left\langle (C_1^2+C_2^2)v,v\right\rangle+\left\langle (D_1^2+D_2^2)v,v\right\rangle\\
		&\Rightarrow\left\langle (C_1-D_1)^2v,v\right\rangle+\left\langle (C_2-D_2)^2v,v\right\rangle=0\\
		&\Rightarrow\Vert (C_1-D_1)v\Vert^2+\Vert (C_2-D_2)v\Vert^2=0\\
		&\Rightarrow B_jv=C_jv=D_jv
	\end{align*}
	for $j=1,2$. This completes the proof.
\end{proof}

\begin{lem}\label{lemma2 for extreme point}
	Let $\mathcal{S}=\left\{(B_1, B_2): B_1, B_2\in(\mathbb{M}_n)_{\text{s.a.}}, B_1^2+B_2^2\leq I_n\right\}$. Then the elements of the set $\mathcal{E}_1(S)=\{(B_1,B_2): B_1, B_2\in(\mathbb{M}_n)_{\text{s.a.}}, B_1^2+B_2^2=I_n\}$ are extreme points of $\mathcal{S}$. 
\end{lem}

\begin{proof}
	Let $(B_1,B_2)\in\mathcal{E}_1(\mathcal{S})$. Assume that $B_1=tC_1+(1-t)D_1, B_2=tC_2+(1-t)D_2$ where $0\leq t\leq 1$ and $(C_1,C_2), (D_1,D_2)\in\mathcal{S}$. Let $\{e_j\}_{j=1}^n$ be an orthonormal basis of $\mathbb{C}^n$. Since $\langle (B_1^2+B_2^2)e_j,e_j\rangle=1$, by Lemma \ref{lemma1 for extreme point}, we have $B_ie_j=C_ie_j=D_ie_j$ for all $1\leq i\leq 2$ and $1\leq j\leq n$. So $B_i=C_i=D_i$ for all $i$. Therefore, $(B_1,B_2)$ is an extreme point of $\mathcal{S}$. 
\end{proof}

\begin{lem}\label{lemma3 for extreme point}
	Let $\mathcal{S}=\left\{(B_1, B_2): B_1, B_2\in(\mathbb{M}_n)_{\text{s.a.}}, B_1^2+B_2^2\leq I_n\right\}$. Then the elements of the set $\mathcal{E}_2(S)$ are extreme points of $\mathcal{S}$, upto unitary similarity, where

	\resizebox{.97\linewidth}{!}{
	\begin{minipage}{\linewidth}
		\begin{align*}
			\mathcal{E}_2(S)&=\left\{(B_1,B_2): B_1=\begin{pmatrix}
				* & X_1 \\
				* & * \\
			\end{pmatrix},\; B_2=\begin{pmatrix}
				* & X_2 \\
				* & * \\
			\end{pmatrix}\in(\mathbb{M}_n)_{\text{s.a.}},\; B_1^2+B_2^2=\begin{pmatrix}
			I_{n-k} &  0\\
			0   & R_k \\
			\end{pmatrix},\; \text{rank}(X_1|X_2)=2k,\;  0<R_k<I_k, \; 1\leq k\leq n\right\}.
		\end{align*}
	\end{minipage}
}
\end{lem}

\begin{proof}
	Let $(B_1,B_2)\in\mathcal{E}_2(\mathcal{S})$. Suppose $1\leq k\leq n$. Take
	\begin{align*}
		B_j=\begin{pmatrix}
			A_j & X_j \\
			X_j^* & Y_j \\
		\end{pmatrix}
	\end{align*}
	 for some $A_j,X_j,Y_j$ where $j=1,2$ such that $B_1^2+B_2^2=\text{diag}(I_{n-k}, R_k),\; \text{rank}(X_1|X_2)=2k$,  and $0<R_k<I_k$. Let $B_1=tC_1+(1-t)D_1$ and $B_2=tC_2+(1-t)D_2$ where $0<t<1$ and $(C_1,C_2),(D_1,D_2)\in\mathcal{S}$. Then, using Lemma \ref{lemma1 for extreme point}, we write
	 \begin{align*}
	 	C_j=\begin{pmatrix}
	 		A_j & X_j \\
	 		X_j^* & Z_j \\
	 	\end{pmatrix},\;
	 	D_j=\begin{pmatrix}
	 	A_j & X_j \\
	 	X_j^* & W_j \\
	 	\end{pmatrix}
	 \end{align*}
	 for some $Z_j,W_j$ where $j=1,2$. Let $E_j=Z_j-Y_j$ and $F_j=W_j-Y_j$ where $j=1,2$. Then $tZ_j+(1-t)W_j=Y_j$ gives $F_j=\frac{t}{t-1}E_j$ for $j=1,2$. Now the condition $C_1^2+C_2^2\leq I_n$ implies that $X_1E_1+X_2E_2=0$, that is, $\begin{pmatrix}
	 	X_1 & X_2 \\
	 \end{pmatrix}\begin{pmatrix}
	 E_1 \\
	 E_2\\
	 \end{pmatrix}=0$. Since $\text{rank}(X_1|X_2)=2k$, we get $E_1=E_2=0$. Therefore, $F_1=F_2=0$. This implies that $Z_j=W_j=Y_j$ for $j=1,2$. Hence $B_j=C_j=D_j$ for $j=1,2$. This shows that $(B_1,B_2)$ is an extreme point of $\mathcal{S}$. This completes the proof.
\end{proof}

\begin{proof}[\textbf{Proof of Theorem \ref{extreme point}}]
	Indeed, by Lemma \ref{lemma2 for extreme point}, \ref{lemma3 for extreme point}, we have $\text{RHS}\subseteq\mathcal{E}(\mathcal{S})$, upto unitary similarity. To prove the other inclusion, let $(B_1,B_2)\in\mathcal{S}$ be such that $(B_1,B_2)\notin\text{RHS}$. Now we consider the following cases.

	\underline{Case I}: Suppose $B_1^2+B_2^2$ is not of full rank. Let $B_1^2+B_2^2=\text{diag}(*,0)$. Then we have
	\begin{align*}
		B_j=\begin{pmatrix}
			A_j & 0 \\
			0 & 0 \\
			\end{pmatrix}
	\end{align*}
	for some $A_j\in(\mathbb{M}_{n-1})_{\text{s.a.}}$ where $j=1,2$. Now observe that 
	\begin{align*}
		B_j=\begin{pmatrix}
			A_j & 0 \\
			0 & 0 \\
		\end{pmatrix}=\frac{1}{2}\begin{pmatrix}
		A_j & 0 \\
		0 & \frac{1}{\sqrt{2}} \\
		\end{pmatrix}+\frac{1}{2}\begin{pmatrix}
		A_j & 0 \\
		0 & -\frac{1}{\sqrt{2}} \\
		\end{pmatrix}
	\end{align*}
	for $j=1,2$. Therefore, $(B_1,B_2)\notin\mathcal{E}(\mathcal{S})$.

	\underline{Case II}: Suppose $B_1^2+B_2^2$ is of full rank. Let $B_1^2+B_2^2=\text{diag}(I_{n-k}, R_k)$ for some $1\leq k\leq n$ where $0<R_k<I_k$. Take 
	\begin{align*}
		B_j=\begin{pmatrix}
			A_j & X_j \\
			X_j^* & Y_j \\
		\end{pmatrix}
	\end{align*}
	for some $A_j,X_j,Y_j$ where $j=1,2$. Since $(B_1,B_2)\notin\mathcal{E}_1(\mathcal{S})\cup\mathcal{E}_2(\mathcal{S})$, we have $\text{rank}(X_1|X_2)<2k$. Now we argue that $(B_1,B_2)\notin\mathcal{E}(\mathcal{S})$. Let
	\begin{align*}
			C_j=\begin{pmatrix}
			A_j & X_j\\
			X_j^* & Y_j+E_j\\
		\end{pmatrix}, \; D_j=\begin{pmatrix}
		A_j & X_j\\
		X_j^* & Y_j-E_j\\
		\end{pmatrix}
	\end{align*}
	where $E_j\in\mathbb{M}_k$ for $j=1,2$. Now we show that there exist $E_1, E_2\neq 0$ such that $(C_1,C_2), (D_1,D_2)\in\mathcal{S}$. It is sufficient to find $E_1,E_2\neq 0$ such that
	\begin{align}
		&X_1E_1+X_2E_2=0,\label{eq1 for extreme point} \\
		&E_1^2+E_2^2+E_1Y_1+Y_1E_1+E_2Y_2+Y_2E_2\leq I_k-R_k \label{eq2 for extreme point}\\
		&E_1^2+E_2^2-E_1Y_1-Y_1E_1-E_2Y_2-Y_2E_2\leq I_k-R_k\label{eq3 for extreme point}
	\end{align}
	Since $\text{rank}(X_1|X_2)<2k$, there exist $E_1,E_2\neq 0$ such that Eqn \ref{eq1 for extreme point} holds. Let $0<t<1$. Now observe that if $(E_1, E_2)$ satisfies Eqn \ref{eq1 for extreme point} then $(tE_1, tE_2)$ also satisfies Eqn \ref{eq1 for extreme point}. So there exists $0<t_\circ<1$ such that $(t_\circ E_1, t_\circ E_2)$ satisfies Eqn \ref{eq1 for extreme point}, \ref{eq2 for extreme point}, \ref{eq3 for extreme point}. Take $E_j=t_\circ E_j$ for $j=1,2$. Then $(C_1, C_2), (D_1, D_2)\in\mathcal{S}$ and $B_j=\frac{1}{2}C_j+\frac{1}{2}D_j$ for $j=1,2$. Therefore, $(B_1,B_2)\notin\mathcal{E}(\mathcal{S})$
\end{proof}

\begin{rmk}
	Let $\mathcal{S}=\left\{(B_1, B_2): B_1, B_2\in(\mathbb{M}_n)_{\text{s.a.}}, B_1^2+B_2^2\leq I_n\right\}$ and $(B_1,B_2)\in\mathcal{S}$. In view of Theorem \ref{extreme point}, we describe a scheme for determining whether $(B_1,B_2)\in\mathcal{E}(\mathcal{S})$. Let $B_1^2+B_2^2=I_{n-k}\oplus R_k$ where $0\leq k\leq n$ and $0<R_k<I_k$. If $k=0$ then $(B_1,B_2)\in\mathcal{E}(\mathcal{S})$. Suppose $1\leq k\leq n$. We write $B_1=\begin{pmatrix}
		* & X_1 \\
		* & * \\
	\end{pmatrix},\; B_2=\begin{pmatrix}
		* & X_2 \\
		* & * \\
	\end{pmatrix}$. If $k>[\frac{n}{3}]$ then $(B_1,B_2)\notin\mathcal{E}(\mathcal{S})$. Let $k\leq[\frac{n}{3}]$. Then  $(B_1,B_2)\in\mathcal{E}(\mathcal{S})\Leftrightarrow \text{rank}(X_1|X_2)=2k$.
\end{rmk}

\begin{rmk}
	It is noteworthy that the extreme points of the set $\{(B_1,B_2): B_1,B_2\in(\mathbb{M}_2)_{\text{s.a.}}, B_1^2+B_2^2\leq I_2\}$ are $\{(B_1,B_2): B_1,B_2\in(\mathbb{M}_2)_{\text{s.a.}}, B_1^2+B_2^2=I_2\}$. However, the extreme points of the set $\mathcal{S}=\{(B_1,B_2): B_1,B_2\in(\mathbb{M}_3)_{\text{s.a.}}, B_1^2+B_2^2\leq I_3\}$ are the union of $\mathcal{E}_1(S)=\{(B_1,B_2): B_1,B_2\in(\mathbb{M}_3)_{\text{s.a.}}, B_1^2+B_2^2=I_3\}$ and $\mathcal{E}_2(S)$ upto unitary similarity where

	\resizebox{.97\linewidth}{!}{
		\begin{minipage}{\linewidth}
			\begin{align*}
				\mathcal{E}_2(S)&=\left\{(B_1,B_2): B_1=\begin{pmatrix}
					* & b_1 \\
					* & * \\
				\end{pmatrix},\; B_2=\begin{pmatrix}
					* & b_2 \\
					* & * \\
				\end{pmatrix}\in(\mathbb{M}_3)_{\text{s.a.}},\; B_1^2+B_2^2=\begin{pmatrix}
					I_2 &  0\\
					0   & r \\
				\end{pmatrix},\; 0<r<1,\; b_1, b_2\in\mathbb{C}^2 \text{ are linearly independent}\right\}.
			\end{align*}
		\end{minipage}
	}\\
	This follows as a particular case of Theorem \ref{extreme point}. Moreover, if $(B_1,B_2)\in\mathcal{E}_2(\mathcal{S})$ and both $B_1,B_2$ have zero diagonal entries, then a straightforward computation shows that $B=\frac{B_1+iB_2}{\sqrt{2}}$ must, upto unitary similarity, be of the following form
	\begin{align}\label{both B_1, B_2 has trace zero}
		B=\begin{pmatrix}
			0 & 0 & t\\
			\sqrt{1-t^2} & 0 & 0\\
			0 & t & 0\\
		\end{pmatrix}
	\end{align}
	where $0<2t^2<1$. This yields Example \ref{counter example 1} by taking $t=\frac{1}{\sqrt{3}}$. It should also be noted that Example \ref{counter example 2} produces $(B_1,B_2)\in\mathcal{E}_2(\mathcal{S})$ such that neither $B_1$ nor $B_2$ has zero diagonal entries.
\end{rmk}

\section{A unital $2$-positive map on the operator system generated by $J_3$}\label{2-positive}

In this section, we study a unital $2$-positive map on the operator system generated by $J_3$ and obtain that the set of all unital contractive maps coincides with the set of all $2$-positive maps on the operator system generated by $J_3$. As a consequence, we describe the $2$nd matricial range of $J_3$. The following is the main result in this section.

\begin{thm}\label{2-positive on J3}
	Let $X\in\mathcal{B}(\mathcal{H})$. Suppose $\varphi:\mathcal{OS}(J_3)\rightarrow\mathcal{OS}(X)$ be a unital self-adjoint preserving linear map with $\varphi(J_3)=X$. Then the following are equivalent.
	\begin{enumerate}
		\item $\varphi$ is $2$-positive.
		\item $\varphi$ is contractive.
		\item $\Vert \alpha X+\beta X^*\Vert\leq\sqrt{\vert\alpha\vert^2+\vert\beta\vert^2}$ for all $\alpha,\beta\in\mathbb{C}$.
	\end{enumerate}
\end{thm}

To prove this, we need the following lemmas.

\begin{lem}\label{lemma1 for 2-positivity}
	Let $B\in\mathbb{M}_n$. Then $I_n\otimes I_3+B\otimes J_3+B^*\otimes J_3^*\geq 0$ if and only if $BB^*+B^*B\leq I_n$.
\end{lem}

\noindent\textit{Proof.}
	We observe that
	\begin{align*}
		I_n\otimes I_3+B\otimes J_3+B^*\otimes J_3^*\geq 0
		&\Leftrightarrow \begin{pmatrix}
			I_n & B & 0\\
			B^* & I_n & B\\
			0 & B^* & I_n\\
		\end{pmatrix}\geq 0\\
		&\Leftrightarrow \begin{pmatrix}
			I_n & B\\
			 B^* & I_n \\
		\end{pmatrix}\geq\begin{pmatrix}
		0 \\
		B \\
		\end{pmatrix}\begin{pmatrix}
		0 & B^*\\
		\end{pmatrix}, \text{ by Schur positivity}\\
		&\Leftrightarrow \begin{pmatrix}
			I_n & B\\
			B^* & I_n-BB^*\\
		\end{pmatrix}\geq 0\\
		&\Leftrightarrow I_n-BB^*\geq B^*B, \text{ by Schur positivity}\\
		&\Leftrightarrow BB^*+B^*B\leq I_n.\tag*{\qed}
	\end{align*}

\begin{lem}\label{boundary}
	Let $X\in\mathcal{B}(\mathcal{H})$ be such that $\Vert \alpha X+\beta X^*\Vert\leq\sqrt{\vert\alpha\vert^2+\vert\beta\vert^2}$ for all $\alpha,\beta\in\mathbb{C}$. Then $I_n\otimes I+B\otimes X+B^*\otimes X^*\geq 0$ whenever $BB^*+B^*B=I_n$ where $B\in\mathbb{M}_n$.
\end{lem}

\begin{proof}
	Indeed, by Theorem \ref{Structure of a dilation of B satisfying BB*+B*B=I}, $B$ has a dilation $\widetilde{B}$ of the form
	\begin{align*}
		\widetilde{B}=\bigoplus_{\substack{\alpha,\beta\in\mathbb{C} \\ \vert\alpha\vert^2+\vert\beta\vert^2=1}} \begin{pmatrix}
			0 & \alpha \\
			\beta & 0 \\
		\end{pmatrix}.
	\end{align*}
	Let $\alpha,\beta\in\mathbb{C}$ with $\vert\alpha\vert^2+\vert\beta\vert^2=1$. Then $\Vert \alpha X+\overline{\beta} X^*\Vert\leq 1$. This implies that $I_n\otimes I+\widehat{B}\otimes X+\widehat{B}^*\otimes X^*\geq 0$ and hence $I_2\otimes I+B\otimes X+B^*\otimes X^*\geq 0$. This completes the proof.
\end{proof}

\begin{lem}\label{lemma2 for 2-positivity}
	Let $\alpha,\beta\in\mathbb{C}$. Then $\Vert \alpha J_3+\beta J_3^*\Vert=\sqrt{\vert \alpha\vert^2+\vert\beta\vert^2}$.
\end{lem}

\noindent\textit{Proof.}
	We compute
	\begin{align*}
		(\alpha J_3+\beta J_3^*)^*(\alpha J_3+\beta J_3^*)=\begin{pmatrix}
			\vert\beta\vert^2 & 0 & \alpha\overline{\beta}\\
			0 & \vert\alpha\vert^2+\vert\beta\vert^2 & 0\\
			\overline{\alpha}\beta & 0 & \vert\alpha\vert^2\\
		\end{pmatrix}\sim\begin{pmatrix}
			\vert\beta\vert^2 & \alpha\overline{\beta} & 0 \\
			\overline{\alpha}\beta & \vert\alpha\vert^2 & 0\\
		   0 & 0 &\vert\alpha\vert^2+\vert\beta\vert^2\\
		\end{pmatrix},
	\end{align*}
where $\sim$ denotes unitary similarity. Therefore, we conclude $\Vert \alpha J_3+\beta J_3^*\Vert^2=\vert\alpha\vert^2+\vert\beta\vert^2$ since
	\begin{align*}
		\left\Vert\begin{pmatrix}
			\vert\beta\vert^2 & \alpha\overline{\beta}\\
			\overline{\alpha}\beta & \vert\alpha\vert^2 \\
		\end{pmatrix}\right\Vert=\left\Vert\begin{pmatrix}
		\overline{\beta}\\
		\overline{\alpha}\\
		\end{pmatrix}\begin{pmatrix}
		\beta & \alpha\\
		\end{pmatrix}\right\Vert=\left\Vert\begin{pmatrix}
		\beta & \alpha\\
		\end{pmatrix}\right\Vert^2=\vert\alpha\vert^2+\vert\beta\vert^2.\tag*{\qed}
	\end{align*}

\begin{proof}[\textbf{\textit{Proof of Theorem \ref{2-positive on J3}}}]
	\textbf{(1)$\Rightarrow$ (2):} This is immediate.\\
	\textbf{(2)$\Rightarrow$ (3):} Let $\alpha,\beta\in\mathbb{C}$. Then $\Vert\alpha X+\beta X^*\Vert\leq\Vert\alpha J_3+\beta J_3^*\Vert=\sqrt{\vert\alpha\vert^2+\vert\beta\vert^2}$, by Lemma \ref{lemma2 for 2-positivity}.\\
	\textbf{(3)$\Rightarrow$ (1):} Let $B\in\mathbb{M}_2$ be such that $I_2\otimes I+B\otimes J_3+B^*\otimes J_3^*\geq 0$. Then, by Lemma \ref{lemma1 for 2-positivity}, we have $BB^*+B^*B\leq I_2$. Therefore, by Theorem \ref{dilation for M2}, $B$ has a dilation $\widetilde{B}\in\mathbb{M}_6$ such that $\widetilde{B}\widetilde{B}^*+\widetilde{B}^*\widetilde{B}=I_6$. So, by Lemma \ref{boundary}, $I_2\otimes I+\widetilde{B}\otimes X+\widetilde{B}^*\otimes X^*\geq 0$ and hence $I_2\otimes I+B\otimes X+B^*\otimes X^*\geq 0$. Therefore, by Lemma \ref{n-positivity}, we conclude that $\varphi$ is $2$-positive.
\end{proof}

\begin{rmk}
	Let $X\in\mathcal{B}(\mathcal{H})$. Consider the unital self-adjoint preserving linear map $\varphi:\mathcal{OS}(J_3)\rightarrow\mathcal{OS}(X)$ such that $\varphi(J_3)=X$. Let $\alpha,\beta,\gamma\in\mathbb{C}$. Since the computation of $\Vert \alpha I_3+\beta J_3+\gamma J_3^*\Vert$ is not straightforward, it is difficult to verify directly from the definition that $\varphi$ is contractive, that is, $\Vert \alpha I+\beta X+\gamma X^*\Vert \leq \Vert \alpha I_3+\beta J_3+\gamma J_3^*\Vert$ for all $\alpha,\beta,\gamma\in\mathbb{C}$. Therefore, Theorem \ref{2-positive on J3} provides a convenient way to verify that $\varphi$ is contractive.
\end{rmk}

We now list down some consequences. The following describes the $2$nd matricial range of the Jordan block of size $3$.

\begin{cor}\label{2-matricial range of J3}
	$W_2(J_3)=\left\{X\in\mathbb{M}_2:\Vert \alpha X+\beta X^*\Vert\leq\sqrt{\vert\alpha\vert^2+\vert\beta\vert^2} \text{ for all } \alpha,\beta\in\mathbb{C}\right\}$.
\end{cor}

\begin{proof}
	This follows from Theorem \ref{2-positive on J3} since any $2$-positive map from $\mathcal{OS}(J_3)$ to $\mathbb{M}_2$ is CP.
\end{proof}

Let $X\in\mathbb{M}_n$. The Frobenius norm of $X$, denoted by $\Vert X\Vert_f$, is defined as $\Vert X\Vert_f=\sqrt{\text{tr}(X^*X)}$.

\begin{cor}\label{trace zero equivalence}
	Let $X\in\mathbb{M}_2$ with $\text{tr}(X)=0$. Then $X\in W_2(J_3)$ if and only if $\Vert X\Vert_f\leq 1$.
\end{cor}

\begin{proof}
	Without loss of generality, we take $X=\begin{pmatrix}
		0 & x \\
		y & 0 \\
	\end{pmatrix}$ where $x,y\in\mathbb{C}$. Let $\alpha,\beta\in\mathbb{C}$. By Cauchy-Schwarz inequality, we have $\vert\alpha x+\beta\overline{y}\vert\leq\sqrt{\vert \alpha\vert^2+\vert\beta\vert^2}\sqrt{\vert x\vert^2+\vert y\vert^2}$. This implies that $\max\{\vert\alpha x+\beta\overline{y}\vert:\vert\alpha\vert^2+\vert\beta\vert^2\leq 1\}=\sqrt{\vert x\vert^2+\vert y\vert^2}$. Therefore, we obtain $\max\{\Vert\alpha X+\beta X^*\Vert:\vert\alpha\vert^2+\vert\beta\vert^2\leq 1\}=\sqrt{\vert x\vert^2+\vert y\vert^2}=\Vert X\Vert_f$. Hence, by Corollary \ref{2-matricial range of J3}, we conclude that $X\in W_2(J_3)$ if and only if $\Vert X\Vert_f\leq 1$.
\end{proof}

Next, we construct an example of a matrix $X\in W_2(J_3)$ for which $\Vert X\Vert_f>1$. This shows that trace zero hypothesis in Corollary \ref{trace zero equivalence} cannot be dropped in general.

\begin{eg}
	Let $u=(\alpha,\beta,\gamma)^t,v=(a,b,c)^t\in\mathbb{R}^3$ be two orthonormal vectors. Consider 
	\begin{align*}
		X=\begin{pmatrix}
			\langle J_3u,u\rangle & \langle J_3v,u\rangle\\
			\langle J_3u,v\rangle & \langle J_3v,v\rangle \\
		\end{pmatrix}=\begin{pmatrix}
		\beta(\alpha+\gamma) & b\alpha+c\beta\\
		a\beta+b\gamma & b(a+c) \\
		\end{pmatrix}.
	\end{align*}
	Now $\Vert X\Vert_f^2=\beta^2(\alpha+\gamma)^2+(b\alpha+c\beta)^2+(a\beta+b\gamma)^2+b^2(a+c)^2:=f(\alpha,\beta,\gamma,a,b,c)$. To construct $X\in W_2(J_3)$ such that $\Vert X\Vert_f>1$, it is sufficient to find two orthonormal vectors $u=(\alpha,\beta,\gamma)^t,v=(a,b,c)^t\in\mathbb{R}^3$ such that $f(\alpha,\beta,\gamma,a,b,c)>1$. Take $u=\frac{1}{\sqrt{2}}(1,-1,0)^t, v=\frac{1}{\sqrt{6}}(1,1,-2)^t$ and check that $u,v$ are two orthonormal vectors with $f(\alpha,\beta,\gamma,a,b,c)=\frac{10}{9}>1$. This produces an example of a matrix $X\in W_2(J_3)$ such that $\Vert X\Vert_f>1$.
\end{eg}

\section{A unital completely positive map on the operator system generated by $J_3$}\label{3-positive}

In this section, we study a unital completely positive map on the operator system generated by $J_3$ and show that the set of all unital completely positive maps coincides with the set of all unital contractive maps in certain cases. As a consequence, we obtain a partial description of the matricial range of $J_3$. The following is the main result of this section.

\begin{thm}\label{CP for J3}
	Let $X\in\mathcal{B}(\mathcal{H})$ with $0\notin W(X)^\circ$. Suppose $\varphi:\mathcal{OS}(J_3)\rightarrow\mathcal{OS}(X)$ is a unital self-adjoint preserving linear map with $\varphi(J_3)=X$. Then the following are equivalent.
	\begin{enumerate}
		\item $\varphi$ is completely positive.
		\item $\varphi$ is contractive.
		\item $\Vert\alpha X+\beta X^*\Vert\leq\sqrt{\vert\alpha\vert^2+\vert\beta\vert^2}$ for all $\alpha,\beta\in\mathbb{C}$.
	\end{enumerate}
\end{thm}

To prove this, we need the following lemmas.

\begin{lem}\label{positive element modified}
	Let $T\in\mathcal{B(\mathcal{H})}$ with $0\in W(T)$. Then every positive element in $\mathbb{M}_n(\mathcal{OS}(T))$ is of the form $A\otimes I+B_1\otimes \Re{(T)}+B_2\otimes \Im{(T)}$ where $A\geq 0$ and $B_1,B_2\in(\mathbb{M}_n)_{\text{s.a.}}$.
\end{lem}

\begin{proof}
	Any positive element in $\mathbb{M}_n(\mathcal{OS}(T))$ is of the form $A\otimes I+B\otimes T+B^*\otimes T^*$ for some $A\geq 0$ and $B\in\mathbb{M}_n$ (by Lemma \ref{form of positive element}). Now observe that
	\begin{align*}
		A\otimes I+B\otimes T+B^*\otimes T^*=A\otimes I+B_1\otimes \Re{(T)}+B_2\otimes \Im{(T)}
	\end{align*}
	where $B_1=B+B^*, B_2=i(B-B^*)$ are self-adjoint. This completes the proof.
\end{proof}

\begin{lem}\label{n-positivity modified}
	Let $T\in\mathcal{B(\mathcal{H})}$ with $0\in W(T)$ and $X\in\mathcal{B(\mathcal{H})}$.  A unital self-adjoint preserving linear map $\varphi:\mathcal{OS}(T)\rightarrow\mathcal{OS}(X)$ with $\varphi(T)=X$ is $n$-positive if and only if $I_n\otimes I+B_1\otimes \Re{(X)}+B_2\otimes \Im{(X)}\geq 0$ whenever $I_n\otimes I+B_1\otimes \Re{(T)}+B_2\otimes \Im{(T)}\geq 0$ where $B_1,B_2\in(\mathbb{M}_n)_{\text{s.a.}}$.
\end{lem}

\begin{proof}
	The proof follows along similar lines to that of Lemma \ref{n-positivity}.
\end{proof}

\begin{lem}\label{self-adjoint structure theorem}
	Let $B_1,B_2\in(\mathbb{M}_n)_{\text{s.a.}}$. Then $I_n\otimes I_3+B_1\otimes\Re{(J_3)}+B_2\otimes \Im{(J_3)}\geq 0$ if and only if $B_1^2+B_2^2\leq 2I_n$.
\end{lem}

\begin{proof}
	Observe that $I_n\otimes I_3+B_1\otimes\Re{(J_3)}+B_2\otimes \Im{(J_3)}=I_n\otimes I_3+B\otimes J_3+B^*\otimes J_3^*$ where$B=\frac{B_1-iB_2}{2}$ and therefore the proof follows by Lemma \ref{lemma1 for 2-positivity}.
\end{proof}

\begin{lem}\label{self-adjoint dilation}
	Let $B_1,B_2\in(\mathbb{M}_n)_{\text{s.a.}}$ be such that $B_1^2+B_2^2=I_n$. Then $B_1, B_2$ have self-adjoint dilations $\widetilde{B}_1, \widetilde{B}_2$ of the form
	\begin{align*}
		\widetilde{B}_1=\bigoplus_{\alpha}\begin{pmatrix}
			0 & \alpha \\
			\overline{\alpha} & 0 \\
		\end{pmatrix},\quad \widetilde{B}_2=\bigoplus_{\beta}\begin{pmatrix}
		0 & \beta \\
		\overline{\beta} & 0 \\
		\end{pmatrix}
	\end{align*}
	where $\vert\alpha\vert^2+\vert\beta\vert^2=1$.
\end{lem}

\begin{proof}
	Let $B=B_1+iB_2$. Then $BB^*+B^*B=2(B_1^2+B_2^2)=2I_n$. Therefore, by Lemma \ref{Structure of a dilation of B satisfying BB*+B*B=I}, $\frac{B}{\sqrt{2}}$ has a dilation $\widetilde{B}$ of the form
	\begin{align*}
		\widetilde{B}=\bigoplus_{\substack{a,b\in\mathbb{C} \\ \vert a\vert^2+\vert b\vert^2=1}} \begin{pmatrix}
			0 & a \\
			b & 0 \\
		\end{pmatrix}.
	\end{align*}
	This implies that $B_1, B_2$ have dilations $\widetilde{B}_1,\widetilde{B}_2$ of the form
	\begin{align*}
		\widetilde{B}_1=\bigoplus_{\substack{a,b\in\mathbb{C} \\ \vert a\vert^2+\vert b\vert^2=1}} \begin{pmatrix}
			0 & \frac{a+\overline{b}}{\sqrt{2}} \\
			\frac{\overline{a}+b}{\sqrt{2}} & 0 \\
		\end{pmatrix},\; \widetilde{B}_2=\bigoplus_{\substack{a,b\in\mathbb{C} \\ \vert a\vert^2+\vert b\vert^2=1}} \begin{pmatrix}
		0 & \frac{a-\overline{b}}{\sqrt{2}i} \\
		\frac{b-\overline{a}}{\sqrt{2}i} & 0 \\
		\end{pmatrix}.
	\end{align*}
	Take $\alpha=\frac{a+\overline{b}}{\sqrt{2}}, \beta=\frac{a-\overline{b}}{\sqrt{2}i}$ and check that $\vert\alpha\vert^2+\vert\beta\vert^2=1$. Hence the conclusion follows.
\end{proof}

\begin{lem}\label{norm}
	Let $X\in\mathcal{B}(\mathcal{H})$. Then $\Vert\alpha X+\beta X^*\Vert\leq\sqrt{\vert\alpha\vert^2+\vert\beta\vert^2}$ for all $\alpha,\beta\in\mathbb{C}$ if and only if $\Vert \alpha \Re{(X)}+\beta \Im{(X)}\Vert\leq\sqrt{\frac{\vert\alpha\vert^2+\vert\beta\vert^2}{2}}$ for all $\alpha,\beta\in\mathbb{C}$.
\end{lem}

\begin{proof}
	The proof follows by writing $\Re{(X)}=\frac{X+X^*}{2}, \text{ and }\Im{(X)}=\frac{X-X^*}{2i}$.
\end{proof}

\begin{lem}\label{boundary case}
	Let $X\in\mathcal{B(\mathcal{H})}$ be such that $\Vert \alpha X+\beta X^*\Vert\leq\sqrt{\vert\alpha\vert^2+\vert\beta\vert^2}$ for all $\alpha,\beta\in\mathbb{C}$. Then $I_n\otimes I+B_1\otimes \Re{(X)}+B_2\otimes \Im{(X)}\geq 0$ whenever $B_1^2+B_2^2=2I_n$ where $B_1,B_2\in(\mathbb{M}_n)_{\text{s.a.}}$. 
\end{lem}

\begin{proof}
	Indeed, by Lemma \ref{self-adjoint dilation}, $B_1,B_2$ have self-adjoint dilations $\widetilde{B}_1,\widetilde{B}_2$ of the form
		\begin{align*}
		\widetilde{B}_1=\bigoplus_{\alpha}\begin{pmatrix}
			0 & \alpha \\
			\overline{\alpha} & 0 \\
		\end{pmatrix},\quad \widetilde{B}_2=\bigoplus_{\beta}\begin{pmatrix}
			0 & \beta \\
			\overline{\beta} & 0 \\
		\end{pmatrix}
	\end{align*}
	where $\vert\alpha\vert^2+\vert\beta\vert^2=2$. Let $\alpha,\beta\in\mathbb{C}$ be such that $\vert\alpha\vert^2+\vert\beta\vert^2=2$. Then $\Vert\alpha X+\beta X^*\Vert\leq \sqrt{2}$. Therefore, by Lemma \ref{norm}, we have $\Vert \alpha \Re{(X)}+ \beta \Im{(X)}\Vert\leq 1$. This implies that $I_n\otimes I+\widetilde{B}_1\otimes \Re{(X)}+\widetilde{B}_2\otimes \Im{(X)}\geq 0$ and hence $I_n\otimes I+B_1\otimes \Re{(X)}+B_2\otimes \Im{(X)}\geq 0$.
\end{proof}

\begin{proof}[\textit{\textbf{Proof of Theorem \ref{CP for J3}}}]
	\textbf{(1)$\Rightarrow$ (2):} This is immediate.\\
	\textbf{(2)$\Rightarrow$ (3):} This follows from Lemma \ref{lemma2 for 2-positivity}.\\
	\textbf{(3)$\Rightarrow$ (1):} Since $0\notin W(X)^\circ$, there exists $\theta\in[0,2\pi)$ such that $\Re{(e^{i\theta}X)}\geq 0$. Take $Y=e^{i\theta}X$. Clearly, we have $\Vert\alpha Y+\beta Y^*\Vert\leq\sqrt{\vert\alpha\vert^2+\vert\beta\vert^2}$ for all $\alpha,\beta\in\mathbb{C}$. Let $Y_1=\Re{(Y)}$ and $Y_2=\Im{(Y)}$. Suppose $B_1,B_2\in(\mathbb{M}_n)_{\text{s.a.}}$ are such that $I_n\otimes I_3+B_1\otimes\Re{(J_3)}+B_2\otimes \Im{(J_3)}\geq 0$. Then, by Lemma \ref{self-adjoint structure theorem}, we have $B_1^2+B_2^2\leq 2I_n$. Let $\mathcal{E}\geq 0$ be such that $B_1^2+B_2^2+\mathcal{E}=2I_n$. Take $B_3=\sqrt{B_1^2+\mathcal{E}}$. Then $-\vert B_3\vert\leq B_1\leq\vert B_3\vert$. Suppose $\mathcal{E}'\geq 0$ is such that $B_1=-\vert B_3\vert+\mathcal{E}'$.

	 We write $I_n\otimes I+B_1\otimes Y_1+B_2\otimes Y_2=I_n\otimes I-\vert B_3\vert\otimes Y_1+B_2\otimes Y_2+\mathcal{E}'\otimes Y_1$. Since $(-\vert B_3\vert)^2+B_2^2=2I_n$, by Lemma \ref{boundary case}, we have $I_n\otimes I-\vert B_3\vert\otimes Y_1+B_2\otimes Y_2\geq 0$ and hence $I_n\otimes I+B_1\otimes Y_1+B_2\otimes Y_2\geq 0$ because $\mathcal{E}', Y_1\geq 0$. Therefore, by Lemma \ref{n-positivity modified}, the unital self-adjoint preserving linear map $\psi:\mathcal{OS}(J_3)\rightarrow\mathcal{OS}(Y)$ with $\psi(J_3)=Y$ is $n$-positive for every $n$, and hence completely positive.

	Take $U=\text{diag}(1, e^{i\theta}, e^{2i\theta})$. Note that $U^*J_3U=e^{i\theta}J_3$. This implies that the unital self-adjoint preserving linear map $\phi:\mathcal{OS}(e^{i\theta} J_3)\rightarrow\mathcal{OS}(J_3)$ with $\phi(e^{i\theta}J_3)=J_3$ is completely positive. Therefore, $\psi\circ\phi:\mathcal{OS}(e^{i\theta} J_3)\rightarrow\mathcal{OS}(Y)$ is completely positive, and hence, by Lemma \ref{affine transformation}, $\varphi$ is completely positive.
\end{proof}

\begin{rmk}
	The condition $0\notin W(X)^\circ$ in Theorem \ref{CP for J3} is equivalent to saying that there exists $\theta\in[0,2\pi)$ such that $\Re{(e^{i\theta}X)}\geq 0$.
\end{rmk}

The following corollary provides a partial description of the matricial range of $J_3$. 

\begin{cor}\label{Matricial range of J3 containment}
	Let $n\in\mathbb{N}$. Then
	\small
	\begin{align*}
		\left\{X\in\mathbb{M}_n: XX^*+X^*X\leq I_n\right\}\subseteq W_n(J_3)\subseteq\left\{X\in\mathbb{M}_n:\Vert\alpha X+\beta X^*\Vert\leq\sqrt{\vert\alpha\vert^2+\vert\beta\vert^2} \text{ for all }\alpha,\beta\in\mathbb{C}\right\}.
	\end{align*}
	\normalsize
\end{cor}

\begin{proof}
	Let $X\in\mathbb{M}_n$ be such that $XX^*+X^*X\leq I_n$. Take $X=X_1+iX_2$ where $X_1,X_2\in(\mathbb{M}_n)_{\text{s.a.}}$. Then $X_1^2+X_2^2\leq\frac{I_n}{2}$. Consider the unital self-adjoint preserving linear map $\varphi:\mathcal{OS}(J_3)\rightarrow\mathcal{OS}(X)$ such that $\varphi(J_3)=X$. Let $B_1,B_2\in(\mathbb{M}_n)_{\text{s.a.}}$ be such that $I_n\otimes I_3+B_1\otimes\Re{(J_3)}+B_2\otimes\Im{(J_3)}\geq 0$. This implies, by Lemma \ref{self-adjoint structure theorem}, that $B_1^2+B_2^2\leq 2I_n$. Observe that $$I_n\otimes I+B_1\otimes X_1+B_2\otimes X_2\geq I_n\otimes I-\vert B_1\vert\otimes \vert X_1\vert+B_2\otimes X_2.$$

	Now we argue that $I_n\otimes I-\vert B_1\vert\otimes \vert X_1\vert+B_2\otimes X_2\geq 0$. Let $X'=\vert X_1\vert+iX_2$. Next we consider the unital self-adjoint preserving linear map $\psi:\mathcal{OS}(J_3)\rightarrow\mathcal{OS}(X')$ such that $\psi(J_3)=X'$. Check that $X'{X'}^*+{X'}^*{X'\leq I_n}$, and therefore $\Vert \alpha X'+\beta {X'}^*\Vert\leq\sqrt{\vert\alpha\vert^2+\vert\beta\vert^2}$ for all $\alpha,\beta\in\mathbb{C}$. This implies, by Theorem \ref{CP for J3}, that $\psi$ is completely positive. Since $(-\vert B_1\vert)^2+B_2^2\leq 2I_n$, by Lemma \ref{self-adjoint structure theorem}, we have $I_n\otimes I_3-\vert B_1\vert\otimes\Re{(J_3)}+B_2\otimes \Im{(J_3)}\geq 0$. Therefore, $I_n\otimes I-\vert B_1\vert\otimes \vert X_1\vert+B_2\otimes X_2\geq 0$ since $\psi$ is completely positive. So we have $I_n\otimes I+B_1\otimes X_1+B_2\otimes X_2\geq 0$. This implies, by Lemma \ref{n-positivity modified}, that $\varphi$ is $n$-positive for every $n$ and hence completely positive. Therefore, $X\in W_n(J_3)$ and we obtain $$\left\{X\in\mathbb{M}_n: XX^*+X^*X\leq I_n\right\}\subseteq W_n(J_3).$$

	Next, let $Y\in W_n(J_3)$. Then there exists a UCP map $\phi:\mathcal{OS}(J_3)\rightarrow\mathcal{OS}(Y)$ such that $\phi(J_3)=Y$. In particular, $\phi$ is $2$-positive and hence, by Theorem \ref{2-positive on J3}, we have $\Vert\alpha Y+\beta Y^*\Vert\leq\sqrt{\vert\alpha\vert^2+\vert\beta\vert^2}$ for all $\alpha,\beta\in\mathbb{C}$. Therefore, we obtain $$W_n(J_3)\subseteq\{X\in\mathbb{M}_n:\Vert\alpha X+\beta X^*\Vert\leq\sqrt{\vert\alpha\vert^2+\vert\beta\vert^2} \text{ for all }\alpha,\beta\in\mathbb{C}\}.$$
	This completes the proof.
\end{proof}

\begin{rmk}
	The extreme points of $W_n(J_3)$ that lie in the convex set $\mathcal{S}=\{X\in\mathbb{M}_n:XX^*+X^*X\leq I_n\}$ remain the extreme points of $\mathcal{S}$. Therefore, Theorem \ref{extreme point} yields the structure of the extreme points of $W_n(J_3)$ that lie in $\mathcal{S}$. 
\end{rmk}

Indeed, the first inclusion in Corollary \ref{Matricial range of J3 containment} is strict since $J_3\in W_n(J_3)$ but $J_3J_3^*+J_3^*J_3\nleq I_3$. However, it remains unclear whether the second inclusion in Corollary \ref{Matricial range of J3 containment} is also strict. At present, we do not have an example demonstrating this. In fact, all our MATLAB/Mathematica computations confirm that the second inclusion is an equality. Therefore, we conclude this section with the following conjecture: the condition $0\notin W(X)^\circ$ in Theorem \ref{CP for J3} can be removed while retaining the same conclusion.

\begin{conj}\label{conjecture 1}
	Let $X\in\mathcal{B}(\mathcal{H})$. Suppose $\varphi:\mathcal{OS}(J_3)\rightarrow\mathcal{OS}(X)$ be a unital self-adjoint preserving linear map with $\varphi(J_3)=X$. Then the following are equivalent.
	\begin{enumerate}
		\item $\varphi$ is completely positive.
		\item $\varphi$ is contractive.
		\item $\Vert \alpha X+\beta X^*\Vert\leq\sqrt{\vert\alpha\vert^2+\vert\beta\vert^2}$ for all $\alpha,\beta\in\mathbb{C}$.
	\end{enumerate}
\end{conj}

Farenick \cite{DF2} proved that every positive map on the operator system of $n\times n$ Toeplitz matrices is completely positive. Therefore, Conjecture \ref{conjecture 1} equivalently states that every unital contractive map on the operator system generated by $J_3$ admits a positive extension to $\mathbb{M}_3$. We are not able to answer this at this stage. We leave this for further investigation in the future.

\section{A unital contractive map on the operator system generated by a $4\times 4$ normal}\label{normal}

In this section, we investigate unital contractive maps on operator systems generated by $4\times 4$ normal matrices. For the purpose of this study, the $4\times 4$ normal matrix can, without loss of generality, be taken in the specific form as described in the following proposition.

\begin{prop}\label{reduction lemma}
	Let $T\in\mathbb{M}_4$ be normal. Then every unital contractive map from $\mathcal{OS}(T)$ to $\mathcal{B(\mathcal{H})}$ is completely positive if and only if every unital contractive map from $\mathcal{OS}(N)$ to $\mathcal{B(\mathcal{H})}$ is completely positive for every $N=\text{diag}(\lambda,-1,i,-i)$ where $\Re{(\lambda)}\geq 0$. 
\end{prop}

\begin{proof}
	$``\Leftarrow"$.  Let $X\in\mathcal{B(\mathcal{H})}$ and $\varphi:\mathcal{OS}(T)\rightarrow\mathcal{OS}(X)$ be a unital contractive map with $\varphi(T)=X$. Without loss of generality, we take $T=\text{diag}(z_1,z_2,z_3,z_4)$ where $z_j\in\mathbb{C}$ for all $j$. Let $S=T-z_4I_4=\text{diag}(w_1,w_2,w_3,0)$ where $w_j=z_j-z_4$ for all $j$ and $Y=X-z_4I$. By Lemma \ref{affine transformation}, the unital self-adjoint preserving linear map $\psi:\mathcal{OS}(S)\rightarrow\mathcal{OS}(Y)$ with $\psi(S)=Y$ is contractive. By the same lemma, it is now sufficient to prove that $\psi$ is completely positive. Let us consider the following cases.

	\underline{Case I}: Let $W(S)$ be a traingle (could be degenerate). Without loss of generality, let $W(S)$ be a traingle with vertices $0,w_1,w_2$. Suppose $S'=\text{diag}(w_1,w_2,0)$. Since $W(S)=W(S')$, by Lemma \ref{equivalent criterion for positivity}, the unital self-adjoint preserving linear maps $\psi':\mathcal{OS}(S)\rightarrow\mathcal{OS}{(S')}$ with $\psi'(S)=S'$ and $\psi'':\mathcal{OS}(S')\rightarrow\mathcal{OS}(S)$ with $\psi''(S')=S$ are both completely positive. This implies that $\psi\circ\psi'':\mathcal{OS}(S')\rightarrow\mathcal{OS}(Y)$ with $\psi\circ\psi''(S')=Y$ is positive and hence, by Theorem \ref{Choi-Li}, completely positive. Therefore $\psi=(\psi\circ\psi'')\circ\psi'$ is completely positive.

	\underline{Case II}: Suppose $W(S)$ is not a triangle. Observe $\mathcal{OS}(S)=\{\text{diag}(a+bw_1+c\overline{w_1},a+bw_2+c\overline{w_2},a+bw_3+c\overline{w_3},a): a,b,c\in\mathbb{C}\}$. Take $a=-1$. We claim that there exist $b,c\in\mathbb{C}$ such that
	\begin{align}\label{1}
		bw_2+c\overline{w_2}=1+i,\quad bw_3+c\overline{w_3}=1-i.
	\end{align}
	Indeed, this is true as the following $b,c$ satisfy Eqns \ref{1}:
	\begin{align}\label{2}
		b=\frac{(1+i)\overline{w_3}-(1-i)\overline{w_2}}{w_2\overline{w_3}-\overline{w_2}w_3},\quad
		c=\frac{(1+i)w_3-(1-i)w_2}{\overline{w_2}w_3-w_2\overline{w_3}}.
	\end{align}
	Note that $w_2\overline{w_3}-\overline{w_2}w_3\neq 0$ because if so then $\frac{w_2}{w_3}=\frac{\overline{w_2}}{\overline{w_3}}=t(\in\mathbb{R})$. That is, $w_2=tw_3$. This implies that $0,w_2,w_3$ are collinear and hence $W(S)$ is a triangle which is not possible. Let $N=aI_4+bS+cS^*=\text{diag}(\lambda,i,-i,-1)$ where $\lambda=a+bw_1+c\overline{w_1}$ with $\Re{(\lambda)}=-1+2\Re{(\frac{w_1(\overline{w_3}-\overline{w_2})}{w_2\overline{w_3}-\overline{w_2}w_3})}$ and $Z=aI+bY+cY^*$. By Lemma \ref{affine transformation}, the unital self-adjoint preserving linear map $\phi:\mathcal{OS}(N)\rightarrow\mathcal{OS}(Z)$ with $\phi(N)=Z$ is contractive. By the same lemma, it is now sufficient to prove that $\phi$ is completely positive. If $\Re{(\lambda)}\geq 0$ then the conclusion follows by the hypothesis. Let $\Re{(\lambda)}<0$. If $W(N)$ is a triangle then also the conclusion follows by Case I. So assume that $W(N)$ is not a triangle. Now we complete the proof considering the following subcases.

	\underline{Subcase I}: Let $\Im{(\lambda)}\geq 0$. Then $I_4-iN=\text{diag}(\mu, 2, 0, 1+i)$ where $\mu=\mu_1+i\mu_2$ such that $\mu_1+\mu_2>2$. Take $M=\text{diag}(\mu, 1+i, 2, 0)$. Then $-I_4+bM+cM^*=\text{diag}(u, i,-i,-1)$ where $u\in\mathbb{C}$ and $b,c$ are the same as given in Eqns \ref{2} (taking $w_1=\mu, w_2=1+i, w_3=2$ in Eqns \ref{2}). Check that $\Re{(u)}\geq 0$ and hence the conclusion follows by the hypothesis.
	
	\underline{Subcase II}: Let $\Im{(\lambda)}\leq 0$. Then $-iI_4-N=\text{diag}(\nu, -2i, 0, 1-i)$ where $\nu=\nu_1+i\nu_2$ such that $\nu_1-\nu_2>2$. Take $R=\text{diag}(\nu, 1-i, -2i, 0)$. Then $-I_4+bR+cR^*=\text{diag}(v, i,-i,-1)$ where $v\in\mathbb{C}$ and $b,c$ are the same as given in Eqns \ref{2} (taking $w_1=\nu, w_2=1-i, w_3=-2i$ in Eqns \ref{2}). Check that $\Re{(v)}\geq 0$ and hence the conclusion follows by the hypothesis.

	$``\Rightarrow"$. This is straightforward.
\end{proof}

The following theorem states that the set of all unital contractive maps coincides with the set of all unital completely positive maps on the operator system generated by $T=\text{diag}(1,-1,i,-i)$.

\begin{thm}\label{main theorem for normal}
	Let $T=\text{diag}(1,-1,i,-i)$ and $X\in\mathcal{B}(\mathcal{H})$. Suppose $\varphi:\mathcal{OS}(T)\rightarrow\mathcal{OS}(X)$ is a unital self-adjoint preserving linear map with $\varphi(T)=X$. Then the following are equivalent.
	\begin{enumerate}
		\item $\varphi$ is completely positive.
		\item $\varphi$ is contractive.
		\item $\Vert \alpha \Re{(X)}+\beta \Im{(X)}\Vert\leq\max\{\vert\alpha\vert,\vert\beta\vert\}$ for all $\alpha,\beta\in\mathbb{C}$.
		\item $w(\widetilde{X})\leq 1$ where $\widetilde{X}=\begin{pmatrix}
			0 & 2\Re{(X)} \\
			2\Im{(X)} & 0\\
		\end{pmatrix}$.
		\item $X=P_1-P_2+i(P_3-P_4)$ where $P_j\geq 0$ for all $1\leq j\leq 4$ and $\sum_{j=1}^4P_j=I$.
		\item $\Re{(X)},\Im{(X)}$ have self-adjoint contractive dilations $A_1,A_2\in\mathcal{B}(\mathcal{H}\oplus\mathcal{H})$ with $A_1A_2=0$.
	\end{enumerate}
\end{thm}

\begin{proof} Let $T_1=\Re{(T)}=\text{diag}(1,-1,0,0),\;  T_2=\Im{(T)}=\text{diag}(0,0,1,-1),\; X_1=\Re{(X)}$ and $X_2=\Im{(X)}$.
	\textbf{(1)$\Rightarrow$ (2):} This is immediate.\\
    \textbf{(2)$\Rightarrow$ (3):} Let $\alpha,\beta\in\mathbb{C}$. Since $\varphi$ is contractive, by Lemma \ref{affine transformation}, we have $\Vert\alpha X_1+\beta X_2\Vert\leq\Vert\alpha T_1+\beta T_2\Vert=\max\{\vert\alpha\vert,\vert\beta\vert\}$.\\
	\textbf{(3)$\Rightarrow$ (4):} Let $f=\begin{pmatrix}
		x\\
		y\\
	\end{pmatrix}\in\mathcal{H}\oplus\mathcal{H}$ be a unit vector. Now we compute 
	\begin{align*}
		\vert\langle \widetilde{X} f,f\rangle\vert\leq 2\left\langle \left(e^{i\theta_1}\Re{(X)}+e^{i\theta_2}\Im{(X)}\right)x,y\right\rangle\leq\left\langle \left(e^{i\theta_1}\Re{(X)}+e^{i\theta_2}\Im{(X)}\right)\frac{x}{\Vert x\Vert},\frac{y}{\Vert y\Vert}\right\rangle
	\end{align*}
	for some $\theta_1,\theta_2\in[0,2\pi)$. Therefore, by $(iii)$, we have $w(\widetilde{X})\leq 1$.\\ 
	\textbf{(4)$\Rightarrow$ (5):} This follows by \cite[Theorem 2.5]{CL2}\\
	\textbf{(5)$\Rightarrow$ (6):} Let $P=P_1+P_2$ and $Q=P_3+P_4$. Clearly, $P,Q\geq 0$ and $P+Q=I$. Note that $-P\leq\Re{(X)}\leq P$ and $-Q\leq\Im{(X)}\leq Q$. This implies that there exist self-adjoint contractions $C_1,C_2\in\mathcal{B}(\mathcal{H})$ such that $\Re{(X)}=\sqrt{P}C_1\sqrt{P}$ and $\Im{(X)}=\sqrt{Q}C_2\sqrt{Q}$. Consider 
	\begin{align*}
		A_1=\begin{pmatrix}
			C_1 & \\
			& 0\\
		\end{pmatrix},\; A_2=\begin{pmatrix}
			0 & \\
			& C_2\\
		\end{pmatrix},\; \text{ and } \; V=\begin{pmatrix}
		\sqrt{P}\\
		\sqrt{Q}\\
		\end{pmatrix}
	\end{align*}
	Check that $V^*V=I, \; \Re{(X)}=V^*A_1V$ and $\Im{(X)}=V^*A_2V$. Therefore, $\Re{(X)},\Im{(X)}$ have self-adjoint contractive dilations $A_1,A_2\in\mathcal{B}(\mathcal{H}\oplus\mathcal{H})$ with $A_1A_2=0$.\\
	\textbf{(6)$\Rightarrow$ (1):} Let $B_1,B_2\in\mathbb{M}_n$ be self-adjoint such that $I_n\otimes I_4+B_1\otimes T_1+B_2\otimes T_2\geq 0$. This implies that $-I_n\leq B_j\leq I_n$ for $j=1,2$. We want to show that $I_n\otimes I+B_1\otimes X_1+B_2\otimes X_2\geq 0$. By $(vii)$, there exist self-adjoint contractive dilations $A_1,A_2\in\mathcal{B}(\mathcal{H\oplus H})$ of $X_1, X_2$ such that $A_1A_2=0$. Since $A_1A_2=0$, we have $\text{ran}A_2\subseteq\text{ker}A_1$. Therefore, writing $\mathcal{H\oplus H}=\text{ker}A_2\oplus\overline{\text{ran}A_2}$, we obtain
	\begin{align*}
		A_1=\begin{pmatrix}
			C_1 & 0\\
			0 & 0\\
		\end{pmatrix},\; A_2=\begin{pmatrix}
			0 & 0\\
			0 & C_2\\
		\end{pmatrix}
	\end{align*}
	where $C_1,C_2$ are self-adjoint, contractions. Since $B_j$ and $C_j$ are self-adjoint contractions, $B_j\otimes C_j$ is also a self-adjoint contraction for all $j$. Therefore, $I_n\otimes I+B_j\otimes C_j\geq 0$ for all $j$. This implies that $I_n\otimes I+B_1\otimes A_1+B_2\otimes A_2\geq 0$. Therefore, $I_n\otimes I+B_1\otimes X_1+B_2\otimes X_2\geq 0$. So, by Lemma \ref{n-positivity modified}, $\varphi$ is $n$-positive for every $n$, and hence completely positive.
\end{proof}

We say an operator $T$ is obtained by an invertible affine transformation to $N$ if $T=f(N,N^*)$ where $f$ is an invertible affine transformation. As a consequence of Theorem \ref{main theorem for normal}, we deduce that the set of all unital completely positive maps coincides with the set of all unital contractive maps on the operator system generated by a normal matrix obtained by an invertible affine transformation to $N=\text{diag}(1,-1,i,-i)$.

\begin{cor}
	Let $T$ be a normal matrix obtained by an invertible affine transformation to $N=\text{diag}(1,-1,i,-i)$. Then every unital contractive map from $\mathcal{OS}(T)$ to $\mathcal{B}(\mathcal{H})$ is completely positive.
\end{cor}

\begin{proof}
The proof is immediate by Lemma \ref{affine transformation}.
\end{proof}

The following example provides a unital, self-adjoint preserving linear map on the operator system generated by a (asymmetric) normal $T\in\mathbb{M}_4$, which is contractive on $\text{span}\{\Re{(T)},\Im{(T)}\}$ but fails to be contractive on $\mathcal{OS}(T)$. This stands in contrast to the case of a unital self-adjoint preserving map on the operator system generated by $J_3$ (recall Theorem \ref{2-positive on J3}).

\begin{eg}
	Let $T=\text{diag}(\frac{1}{2},-1,i,-i)$ and $X=\begin{pmatrix}
		\frac{1}{4}& 0.4i\\
		0.4i & -\frac{1}{\sqrt{2}}\\
	\end{pmatrix}$. Consider the unital self-adjoint preserving linear map $\varphi:\mathcal{OS}(T)\rightarrow\mathcal{OS}(X)$ with $\varphi(T)=X$.  Since $\Vert\beta\Re{(X)}+\gamma\Im{(X)}\Vert_f\leq\max\{\vert\beta\vert,\vert\gamma\vert\}$ for all $\beta,\gamma\in\mathbb{C}$. So $\varphi$ is contractive on $\text{span}\{\Re{(T)},\Im{(T})\}$ (also on $\text{span}\{I_4,\Re{(T)}\}$ and $\text{span}\{I_4,\Im{(T)}\}$). But $\varphi$ is not contractive on $\mathcal{OS}(T)$ (and hence not completely positive). Because choose $\beta=3.79+0.2i,\gamma=0.1-2.67945i$ and check that $\Vert I+\beta\Re{(X)}+\gamma\Im{(X)}\Vert>2.8974>2.8968>\Vert I+\beta\Re{(T)}+\gamma\Im{(T)}\Vert$. 
\end{eg}

An intriguing question that naturally arises at this stage is the following:

\begin{que}\label{conjecture 2}
	Let $N=\text{diag}(\lambda,-1,i,-i)$ with $\Re{(\lambda)}\geq 0$. Is every unital contractive map from $\mathcal{OS}(N)$ to $\mathcal{B}(\mathcal{H})$ completely positive?
\end{que}

In light of Proposition \ref{reduction lemma}, Question \ref{conjecture 2} equivalently states that, given a normal matrix $T\in\mathbb{M}_4$, does every unital contractive map on $\mathcal{OS}(T)$ admit a positive extension to $\mathbb{M}_4$? We are not able to answer this at this point. We leave this for further investigation in the future. However, it should be noted that this does not hold in general for a unital contractive map on an arbitrary operator system within a $C^*$-algebra. The reader may refer \cite{PR} for examples.

\section{Remarks and Examples}\label{remarks and examples}

In this section, we discuss some remarks on the matricial range of $J_n$ and present an example of a unital contractive map on the operator system generated by a $3\times 3$ matrix which is not completely positive. In addition, we provide several examples illustrating both contractive and non-contractive maps on operator systems generated by normal matrices and derive some applications on the constrained unitary dilations of $J_2$.

The matricial range of $J_n$ is invariant under transposition and rotation. To discuss this, we begin with the following lemma. Given $T\in\mathcal{B}(\mathcal{H})$, $T^t$ denotes the transpose of $T$ with respect to a fixed orthonormal basis of $\mathcal{H}$.

\begin{lem}\label{transpose map is positive}
	Let $T\in\mathcal{B}(\mathcal{H})$. Then $\overline{W(T)}=\overline{W(T^t)}$.
\end{lem}

\begin{proof}
	Let $\varphi:\mathcal{OS}(T)\rightarrow\mathcal{OS}(T^t)$ be a unital self-adjoint preserving map with $\varphi(T)=T^t$. Indeed, $\varphi$ is positive and hence the proof follows by Lemma \ref{equivalent criterion for positivity}.
\end{proof}

\begin{rmk}
	Let $k\in\mathbb{N}$. Then $W_k(J_n)$ is transpose-invariant.
\end{rmk}

\begin{proof}
	Let $X\in W_k(J_n)$ and $B\in\mathbb{M}_k$. Then, by Theorem \ref{description of the matricial range}, we have $W(B^t\otimes X)\subseteq W(B^t\otimes J_n)$. Since $J_n$ is unitarily similar to $J_n^t$, we obtain $W(B^t\otimes J_n)=W(B^t\otimes J_n^t)$. Therefore, $W(B^t\otimes X)\subseteq W(B^t\otimes J_n^t)$ and hence, by Lemma \ref{transpose map is positive}, $W(B\otimes X^t)\subseteq W(B\otimes J_n)$ for all $B\in\mathbb{M}_k$. This implies, by Theorem \ref{description of the matricial range}, that $X^t\in W_k(J_n)$. This shows that $W_k(J_n)$ is transpose-invariant.
\end{proof}

Next, we remark that the matricial range of $J_n$ is rotation invariant. Although this follows from \cite[Theorem 2.1]{LT}, we discuss a different proof for the reader's convenience.



\begin{rmk}
	Let $k\in\mathbb{N}$. Then $W_k(J_n)$ is rotation-invariant.
\end{rmk}

\begin{proof}
	Let $X\in W_k(J_n)$ and $\theta\in[0,2\pi)$. Take $U=\text{diag}(1, e^{i\theta},\hdots,e^{i(n-1)\theta})$. Note that $U^*J_nU=e^{i\theta} J_n$. Therefore, by (\hyperref[P2]{P2}) and (\hyperref[P3]{P3}), we have $W_k(J_n)=e^{i\theta}W_k(J_n)$. This implies that $e^{i\theta}X\in W_k(J_n)$. Hence $W_k(J_n)$ is rotation-invariant.
\end{proof}


Now we present an example of a unital $2$-positive map (and hence contractive) on the operator system generated by a $3\times 3$ matrix which is not completely positive. To discuss this, we need the following lemmas.

\begin{lem}\label{transpose equivalent}
	Let $B\in\mathbb{M}_2$. Then $B$ is unitarily similar to $B^t$.
\end{lem}

\begin{proof}
	Without loss of generality, let $B=\begin{pmatrix}
		r & x\\
		y & r\\
	\end{pmatrix}$ for some $r,x,y\in\mathbb{C}$. Take $U=\begin{pmatrix}
		0 & 1\\
		1 & 0\\
	\end{pmatrix}$ and check that $U^*BU=B^t$. So $B$ is unitarily similar to $B^t$.
\end{proof}

\begin{lem}\label{transpose map}
	Let $T\in\mathcal{B(\mathcal{H})}$. A unital self-adjoint preserving map $\varphi:\mathcal{OS}(T)\rightarrow\mathcal{OS}(T^t)$ with $\varphi(T)=T^t$ is $2$-positive.
\end{lem}

\begin{proof}
	Let $\lambda\in W(T)$. Take $S=T-\lambda I$. Then $0\in W(S)$. By Lemma \ref{affine transformation}, it is now sufficient to prove the unital self-adjoint preserving linear map $\psi:\mathcal{OS}(S)\rightarrow\mathcal{OS}(S^t)$ with $\psi(S)=S^t$ is $2$-positive. Let $B\in\mathbb{M}_2$ such that $I_2\otimes I+B\otimes S+B^*\otimes S^*\geq 0$. Now
	\begin{align*}
		I_2\otimes I+B\otimes S+B^*\otimes S^*\geq 0
		&\Rightarrow I_2\otimes I+B^t\otimes S+{B^*}^t\otimes S^*\geq 0, \text{ using Lemma } \ref{transpose equivalent}\\
		&\Rightarrow I_2\otimes I+B\otimes S^t+{B^*}\otimes {S^t}^*\geq 0
	\end{align*}
	Therefore, by Lemma \ref{n-positivity}, $\psi$ is $2$-positive and hence $\varphi$ is $2$-positive.
\end{proof}

\begin{eg}\label{weighted jordan block}(\cite[Example 3.6]{CL2})
	Consider
	\begin{align*}
		T=\begin{pmatrix}
			0 & 0& 0\\
			1 & 0 & 0\\
			0 & 2 & 0\\
		\end{pmatrix}.
	\end{align*}
	Let $\varphi:\mathcal{OS}(T)\rightarrow\mathcal{OS}(T^t)$ be a unital self-adjoint preserving linear map with $\varphi(T)=T^t$. By Lemma \ref{transpose map}, $\varphi$ is $2$-positive. But $\varphi$ is not completely positive. Because if so, by Kadison-Schwarz inequality, we have $\varphi(2TT^*+T^*T)\geq 2T^t(T^t)^*+(T^t)^*T^t\geq 0$. This implies that $9=\Vert2T^t(T^t)^*+(T^t)^*T^t\Vert\leq \Vert2TT^*+T^*T\Vert=8$. This is a contradiction. Hence $\varphi$ is not completely positive. This produces an example of a $2$-positive map (and hence contractive) on the operator system generated by a $3\times 3$ matrix which is not completely positive. 
\end{eg}

Next, we discuss a few examples of contractive and non-contractive maps on operator systems generated by normal matrices and derive some applications on the constrained unitary dilations of $J_2$. Recall, from Lemma \ref{observation}, that $J_2$ has no unitary dilation $U$ satisfying $-\frac{1}{2}\leq\Re{(U)}\leq\frac{1}{2}$. The following example provides a unitary dilation $U$ of $J_2$ satisfying $-\frac{1}{\sqrt{2}}\leq\Re{(U)}\leq\frac{1}{\sqrt{2}}$.

\begin{eg}\label{eg1}
	Let $T=\text{diag}(e^{-i\frac{\pi}{4}}, e^{i\frac{\pi}{4}}, e^{i\frac{3\pi}{4}},e^{-i\frac{3\pi}{4}})$. Consider the unital self-adjoint preserving map $\varphi:\mathcal{OS}(T)\rightarrow\mathcal{OS}(J_3)$ with $\varphi(T)=J_3$. Take $N=e^{i\frac{\pi}{4}}T=\text{diag}(1,i,-1,-i)$ and $X=e^{i\frac{\pi}{4}}J_3$. Let $\psi:\mathcal{OS}(N)\rightarrow\mathcal{OS}(X)$ be the unital self-adjoint preserving map with $\psi(N)=X$. We claim that $\psi$ is completely positive. Let $B\in\mathbb{M}_n$ such that $I_n\otimes I_4+B\otimes T+B^*\otimes T^*\geq 0$. This implies that $-\frac{1}{2}\leq\Re{(B)}\leq\frac{1}{2}$ and $-\frac{1}{2}\leq\Im{(B)}\leq\frac{1}{2}$. So $BB^*+B^*B=2(\Re{(B)}^2+\Im{(B)}^2)\leq 1$. Therefore, by Lemma \ref{lemma1 for 2-positivity}, $I_n\otimes I_2+B\otimes X+B^*\otimes X^*\geq 0$. So, by Lemma \ref{n-positivity}, $\psi$ is $n$-positive for every $n$ and hence completely positive. This implies that $\varphi$ is completely positive. So, by Proposition \ref{equivalent criterion for CP}, $U=T\otimes I$ is a unitary dilation of $J_3$ satisfying $-\frac{1}{\sqrt{2}}\leq\Re{(U)}\leq\frac{1}{\sqrt{2}}$. In particular, $U=T\otimes I$ is a unitary dilation of $J_2$ satisfying $-\frac{1}{\sqrt{2}}\leq\Re{(U)}\leq\frac{1}{\sqrt{2}}$. 
\end{eg}

Now we will discuss a few examples of non-contractive maps on operator systems generated by normal matrices. As a consequence, we deduce that $J_2$ has no unitary dilation $V\in\mathbb{M}_n$ satisfying $-t\leq\Re{(V)}\leq t$ for every $0\leq t<\frac{1}{\sqrt{2}}$.

\begin{eg}\label{eg2}
	Let $\theta_1,\hdots,\theta_n\in[0,2\pi)$ such that $\vert\cos\theta_j\vert<\frac{1}{\sqrt{2}}$ for all $j$ and $T=\text{diag}(e^{i\theta_1},\hdots, e^{i\theta_n})$. Consider the unital self-adjoint preserving map $\varphi:\mathcal{OS}(T)\rightarrow\mathcal{OS}(J_2)$ with $\varphi(T)=J_2$. We will prove that $\varphi$ is not contractive. To see this, we compute
	\begin{align*}
		&\Vert I_2+\beta J_2+\gamma J_2^*\Vert^2=1+\frac{\beta^2+\gamma^2}{2}+\sqrt{\frac{(\beta^2-\gamma^2)^2}{4}+(\beta+\gamma)^2},\\
		&\Vert I_2+\beta T+\gamma T^*\Vert^2=\max_{1\leq j\leq n} \{1+\beta^2+\gamma^2+2(\beta+\gamma)\cos\theta_j+2\beta\gamma\cos 2\theta_j\}
	\end{align*}
	for all $\beta,\gamma>0$. Since $\vert\cos\theta_j\vert<\frac{1}{\sqrt{2}}$ for all $j$, we have
	\begin{align*}
		\Vert I_2+\beta T+\gamma T^*\Vert^2
		&<1+\beta^2+\gamma^2+\sqrt{2}(\beta+\gamma)+2\beta\gamma\cos 2\theta\\
		&=1+\beta^2+\gamma^2+\sqrt{2}(\beta+\gamma)-2\beta\gamma(1-2\cos^2\theta)
	\end{align*}
	where $\cos 2\theta=\min\{\cos2\theta_j:1\leq j\leq n\}$. Let $n\in\mathbb{N}$ such that $\frac{1}{10^n}<1-2\cos^2\theta$. So it is sufficient to find $\beta,\gamma>0$ such that
	\begin{align*}
		\text{LHS}=1+\beta^2+\gamma^2+\sqrt{2}(\beta+\gamma)-\frac{2\beta\gamma}{10^n}<\Vert I_2+\beta J_2+\gamma J_2^*\Vert^2=\text{RHS}.
	\end{align*}
	Choose $\beta=10^{5n}$ and $\gamma=10^n$ and check that LHS<RHS. Therefore, $\varphi$ is not contractive (and hence not completely positive).
\end{eg}

The following remark strengthens Lemma \ref{observation}. 

\begin{rmk}
	$J_2$ has a unitary dilation $U$ satisfying $-\frac{1}{\sqrt{2}}\leq\Re{(U)}\leq\frac{1}{\sqrt{2}}$. Furthermore, $J_2$ has no unitary dilation $V\in\mathbb{M}_n$ satisfying $-t\leq\Re{(V)}\leq t$ for every $0\leq t<\frac{1}{\sqrt{2}}$. The proof follows from Examples \ref{eg1} and \ref{eg2}. 
\end{rmk}

Finally, we conclude with the following example.

\begin{eg}
	Let $\theta_1,\hdots,\theta_n\in[0,2\pi)$ such that $\vert\cos\theta_j\vert>\frac{1}{\sqrt{2}}$ for all $j$ and $T=\text{diag}(e^{i\theta_1},\hdots, e^{i\theta_n})$. Consider the unital self-adjoint preserving map $\varphi:\mathcal{OS}(T)\rightarrow\mathcal{OS}(J_2)$ with $\varphi(T)=J_2$. Then $\varphi$ is not contractive (and hence not completely positive). Because if so then the unital self-adjoint preserving map $\psi:\mathcal{OS}(N)\rightarrow\mathcal{OS}(J_2)$ with $\varphi(N)=J_2$ is also contractive, where $N=iT$ (since $\Vert I_2+\beta J_2+\gamma J_2^*\Vert=\Vert I_2+i\beta J_2-i\gamma J_2^*\Vert$). This contradicts Example \ref{eg2}.
\end{eg}

\section{Acknowledgments}
The authors sincerely thank the anonymous reviewer for the careful reading of the manuscript and for the valuable suggestions, which have significantly improved the presentation of the paper. The first-named author is extremely grateful to Chi-Kwong Li for insightful discussions on Section \ref{normal} during the CIMPA Workshop on Linear Preserver Problems held at IISER Bhopal. This author's research is supported by the Institute Postdoctoral Fellowship of Indian Institute of Technology Bombay. The research of the second-named author is supported by the UGC through a Senior Research Fellowship.

\bibliographystyle{acm}
\bibliography{Bibliography}	

 \end{document}